\documentclass[reqno,11pt]{preprint}

\usepackage[full]{textcomp}
\usepackage{times}

\usepackage[cal=boondoxo]{mathalfa}
\usepackage{microtype}

\usepackage{amsmath}
\usepackage{amsfonts}
\usepackage{amssymb}

\usepackage[all]{xy}
\usepackage[utf8]{inputenc} 
\usepackage{comment}

\usepackage{hyperref}
\usepackage{breakurl}

\usepackage{mathrsfs}
\usepackage{enumitem}

\usepackage{tikz}
\usepackage{dsfont}
\usepackage{mhequ} 
\usepackage{mhsymb} 
\usepackage{mhenvs} 
\usepackage{array}

\usepackage{wasysym}
\usepackage{centernot}
\usepackage{mathtools}

\setlist{noitemsep,nolistsep,leftmargin=1.7em}
\usetikzlibrary{snakes}
\usetikzlibrary{decorations}
\usetikzlibrary{positioning}
\usetikzlibrary{shapes}
\usetikzlibrary{external}
\DeclareFontFamily{U}{mathx}{\hyphenchar\font45}
\DeclareFontShape{U}{mathx}{m}{n}{
      <5> <6> <7> <8> <9> <10>
      <10.95> <12> <14.4> <17.28> <20.74> <24.88>
      mathx10
      }{}
\DeclareSymbolFont{mathx}{U}{mathx}{m}{n}
\DeclareMathSymbol{\bigtimes}{1}{mathx}{"91}

\def\s{\mathfrak{s}}

\definecolor{darkred}{rgb}{0.7,0.1,0.1}
\definecolor{darkblue}{rgb}{0.1,0.1,0.8}
\definecolor{darkgreen}{rgb}{0.1,0.7,0.1}
\providecommand{\figures}{false}
{ \ifthenelse{\equal{\figures}{false}} {#1}{\[ {\rm Figure \ missing !} \]} }{}
\def\id{\mathrm{id}}

\def\CH{\mathcal{H}}

\def\CJ{\mathcal{J}}
\def\CA{\mathcal{A}}

\def\CC{\mathcal{C}}

\def\CB{\mathcal{B}}
\def\CM{\mathcal{M}}
\def\CT{\mathcal{T}}

\def\K{\mathfrak{K}}

\def\Labe{\mathfrak{e}}

\def\Labn{\mathfrak{n}}

\def\Labhom{\mathfrak{t}}
\def\Lab{\mathfrak{L}}

\def\Deltap{\Delta^{\!+}}

\def\${|\!|\!|}

\makeatletter
\newcommand*\bigcdot{\mathpalette\bigcdot@{.5}}
\newcommand*\bigcdot@[2]{\mathbin{\vcenter{\hbox{\scalebox{#2}{$\m@th#1\bullet$}}}}}
\makeatother

\def\CR{\mathcal{R}}
\def\CF{\mathcal{F}}

\newenvironment{DIFnomarkup}{}{} 

\newtheorem{assumption}{Assumption}

\newfont{\indic}{bbmss12}

\def\PPi{\boldsymbol{\Pi}}

\def\mail#1{\burlalt{#1}{mailto:#1}}

\colorlet{symbols}{blue!90!black}
\colorlet{testcolor}{green!60!black}
\colorlet{connection}{red!30!black}
\def\symbol#1{\textcolor{symbols}{#1}}
\def\symbol#1{\textcolor{symbols}{#1}}

\tikzset{
root/.style={circle,fill=black!50,inner sep=0pt, minimum size=3mm},
        circ/.style={circle,fill=white,draw=black,very thin,inner sep=.5pt, minimum size=1.2mm},
        dot/.style={circle,fill=black,inner sep=0pt, minimum size=1.2mm},
        dotred/.style={circle,fill=black!50,inner sep=0pt, minimum size=2mm},
        var/.style={circle,fill=black!10,draw=black,inner sep=0pt, minimum size=3mm},
        kernel/.style={semithick,shorten >=2pt,shorten <=2pt},
        kernel1/.style={thick},
        kernels/.style={snake=zigzag,shorten >=2pt,shorten <=2pt,segment amplitude=1pt,segment length=4pt,line before snake=2pt,line after snake=5pt,},
		kernels1/.style={snake=zigzag,segment amplitude=0.5pt,segment length=2pt},
		rho1/.style={densely dotted,semithick},
        rho/.style={densely dashed,semithick,shorten >=2pt,shorten <=2pt},
           testfcn/.style={dotted,semithick,shorten >=2pt,shorten <=2pt},
        renorm/.style={shape=circle,fill=white,inner sep=1pt},
        labl/.style={shape=rectangle,fill=white,inner sep=1pt},
        xic/.style={very thin,circle,fill=symbols,draw=black,inner sep=0pt,minimum size=1.2mm},
        xi/.style={very thin,circle,fill=blue!10,draw=black,inner sep=0pt,minimum size=1.2mm},
	xib/.style={very thin,circle,fill=blue!10,draw=black,inner sep=0pt,minimum size=1.6mm},
	xie/.style={very thin,circle,fill=green!50!black,draw=black,inner sep=0pt,minimum size=1mm},
	xid/.style={very thin,circle,fill=symbols,draw=black,inner sep=0pt,minimum size=1.6mm},
	edgetype/.style={very thin,circle,draw=black,inner sep=0pt,minimum size=5mm},
	nodetype/.style={very thick,circle,draw=black,inner sep=0pt,minimum size=5mm},
	kernels2/.style={very thick,draw=connection,segment length=12pt},
clean/.style={thin,circle,fill=black,inner sep=0pt,minimum size=1mm},	not/.style={thin,circle,fill=symbols,draw=connection,fill=connection,inner sep=0pt,minimum size=0.8mm},
	>=stealth,
        }

\tikzset{ individus/.style={scale=0.40,draw,circle,thick,fill=black!10},
 individu/.style={scale=0.40,draw,circle,thick,fill=black!50},       } 
\makeatletter
\def\DeclareSymbol#1#2#3{\expandafter\gdef\csname MH@symb@#1\endcsname{\tikz[baseline=#2,scale=0.15,draw=symbols]{#3}}\expandafter\gdef\csname MH@symb@#1s\endcsname{\scalebox{0.7}{\tikz[baseline=#2,scale=0.15,draw=symbols]{#3}}}}
\def\<#1>{\csname MH@symb@#1\endcsname}
\makeatother

 \def\1{\mathbf{\symbol{1}}}

\def\one{\mathbf{1}}
\def\eps{\varepsilon}

\DeclareMathAlphabet{\mathpzc}{OT1}{pzc}{m}{it}

\def\one{\mathbf{1}}
\def\eps{\varepsilon}

\def\Deltap{\Delta^{\!+}}

\def\PPi{\boldsymbol{\Pi}}

\def\id{\mathrm{id}}

\def\simnot{\stackrel{\vbox to 0.15em{\hbox{\kern0.07em$^\circ$}}}{\sim}}

\begin{document}

\title{Bogoliubov type recursions for\\ renormalisation in regularity structures}

\author{Yvain Bruned$^1$, Kurusch Ebrahimi-Fard$^2$}
\institute{ 
 IECL (UMR 7502), Université de Lorraine
 \and Norwegian University of Science and Technology NTNU \\
Email:\ \begin{minipage}[t]{\linewidth}
\mail{yvain.bruned@univ-lorraine.fr},
\\
\mail{kurusch.ebrahimi-fard@ntnu.no}.
\end{minipage}}

\maketitle

\begin{abstract}
Hairer's regularity structures transformed the solution theory of singular stochastic partial differential equations. The notions of positive and negative renormalisation are central and the intricate interplay between these two renormalisation procedures is captured through the combination of cointeracting bialgebras and an algebraic Birkhoff-type decomposition of bialgebra morphisms. This work revisits the latter by defining Bogoliubov-type recursions similar to Connes and Kreimer's formulation of BPHZ renormalisation. We then apply our approach to the renormalisation problem for SPDEs.
\end{abstract}

\setcounter{tocdepth}{2}
\tableofcontents


\section{Introduction}
\label{sec::1}

The theory of Regularity Structures (RS) has been developed to its full generality within a few years since its initial presentation by Hairer \cite{reg}. Thanks to recent progress, one has local well-posedness results for a large class of singular Stochastic Partial Differential Equations (SPDEs). This achievement relies, among others, on the following papers \cite{BHZ,ajay,BCCH}. It culminated in the construction of a natural random dynamic on the space of loops in a Riemannian manifold described in \cite{BGHZ} and the Langevin dynamic for the 2D Yang--Mills measure in \cite{CCHS}. Friz and Hairer \cite{FrizHai} present a textbook introduction to RS and \cite{EMS,BaiHos} give short respectively extended surveys on these developments. The algebraic foundation of the theory has been developed in \cite{BHZ}, where two renormalisation procedures are shown to be in cointeraction: the first recenters distributions around a point such that they can be understood as recentered monomials. The second renormalisation cures divergences coming from ill-defined distributional products in a singular SPDE. In the abstract of reference \cite{BHZ}, the key parts of these renormalisations have been highlighted: ``Two twisted antipodes play a fundamental role in the construction and provide a variant of the algebraic Birkhoff factorisation".

The main contribution of this paper is to make this link more precise and to explore the extent to which Bogoliubov's recursions and hence the algebraic Birkhoff factorisation are altered in the context of the renormalisation problem of SPDEs. We establish a new Birkhoff factorisation different from the one introduced by Connes-Kreimer in \cite{CKI}. Birkhoff factorisations are essential objects that appear in many different fields such as numerical analysis where it is used in \cite{BS}  for the   local error analysis of low regularity schemes for dispersive PDEs. We expect to see them used again in the context of (S)PDEs and this work provides new tools in this direction.
   
Let us outline the paper by summarising the content of its sections. In Section~\ref{sec::2}, we recall the algebraic Birkhoff factorisation following Connes and Kreimer \cite{CKI}, which unveiled an elegant group theoretical formulation of the BPHZ renormalisation procedure in perturbative quantum field theory \cite{BP,Hepp,Zimmermann}. We refer the reader to  \cite{CasKen82,Collins84,Zavialov,panzer} for useful references on renormalisation in perturbative quantum field theory. Bogoliubov's recursions for counterterms and renormalised amplitudes are characterised as solutions of a factorisation problem in the group of characters over a (specific Feynman graph or rooted tree) Hopf algebra. Then we introduce a factorisation-type renormalisation with the coproduct replaced by a coaction. The counterterm recursion is defined via a comodule structure. These two structures are considered for a connected Hopf algebra. We introduce also Taylor-jet operators forming a (family of) Rota--Baxter map(s), which is central for the Bogoliubov recursions to solve the factorisation problem. Section~\ref{sec::Birkhoff} contains the new and important results. We consider decorated trees as they appear in regularity structures and define a comodule-Hopf algebra structure on them. This structure has been originally introduced in \cite{BHZ}. We present its algebraic construction and postpone the link to SPDEs to the next section. The main difficulty lies in the fact that the Hopf algebra at play is not connected. Therefore, the results presented in Section~\ref{sec::2} cannot be applied directly if one wants to set up a Birkhoff-type factorisation. This problem is circumvented by defining a modified reduced coproduct and use a family of Rota--Baxter maps in order to give one of the main definitions of this paper (Definition~\ref{Bogoliubov}). Then, by exploiting the Rota--Baxter property one can show the main result (Theorem \ref{main theorem}). It is also shown that under certain assumptions, the recentering map does not depend on a priori recentering of the polynomials. These results have to be understood as an alternative way of defining the notion of model, which forms a critical part in Hairer's theory \cite{reg}. In Section~\ref{sec::4}, we present two Birkhoff factorisations connected to singular SPDEs. The first one concerns the construction of the recentering map which crucially relies on the main result of Section~\ref{sec::Birkhoff}. Then, we present the notion of negative renormalisation, which is close in spirit to the approach outlined in Section~\ref{sec::2}, where the comodule structure is used.


\section{Algebraic Birkhoff factorisation}
\label{sec::2}

Connes and Kreimer discovered a Hopf algebraic formulation of the renormalisation process in perturbative quantum field theory \cite{CKI}. It permits to capture the so-called BPHZ subtraction method  \cite{Collins84,Zavialov}  in terms of an algebraic Birkhoff decomposition of Feynman rules seen as an element in the group of Hopf algebra characters. The factors in this decomposition give the renormalised respectively counterterm parts. 

The notion of renormalisation in regularity structures permits as well a Hopf algebraic formulation. However, the tree Hopf algebras at play are not necessarily connected. Moreover, renormalisation can not be described as decomposition at a group theoretical level. Instead, we will have to consider a variant of this approach using a comodule structure.

\smallskip

Let $H=\C \mathbf{1} \oplus \bigoplus_{n>0} H^{(n)}$ be a connected graded Hopf algebra over $ \C $ with coproduct $ \Delta : H \to H \otimes H  $, antipode $\CA : H \to H $ and counit $\mathbf{1}^* : H \to \C$. Recall that the latter is a linear map which equals $1$ on $ \one$ and zero else. In the sequel, we will use Sweedler's notation in order to describe the coproduct, $\Delta$, as well as the corresponding reduced coproduct, $\Delta\!'$:
\allowdisplaybreaks
\begin{equs}
	 \Delta \tau = \sum_{(\tau)} \tau^{(1)} \otimes \tau^{(2)} 
	 		 &= \Delta\!' \tau + \tau \otimes \one + \one \otimes \tau \\
			 &= \sideset{}{\!^{'}}\sum_{(\tau)} \tau' \otimes \tau'' + \tau \otimes \one + \one \otimes \tau.
\end{equs}

\begin{remark} \label{rmk:BCK}
Let $\mathcal{RT}$ denote the set of non-planar rooted trees and $\mathcal{T}:=\langle \mathcal{RT}\rangle$ is the corresponding space. A tree is naturally graded by its number of vertices. The Butcher--Connes--Kreimer Hopf algebra of rooted trees, $H_\mathcal{T}$, provides a key example of a connected graded Hopf algebra -- of combinatorial nature. It plays an important role in the theory of Butcher's B-series \cite{MR2657947} in numerical analysis. Connes and Kreimer studied $H_\mathcal{T}$ in great detail in the context of renormalisation in perturbative quantum field theory \cite{CK}. They defined its coproduct using the notion of admissible cuts on rooted trees. Moreover, they also described a recursive formula based on the fact that any rooted tree $\tau \in \mathcal{T}$, different from the empty tree, $\one$, can be written in terms of the $B_+$-operator. Indeed, in terms of the latter, we have that $\tau = B_+(\tau_1 \cdots \tau_n)$, which connects the roots of the trees in the forest $\tau_1 \cdots \tau_n \in H_\mathcal{T}$ to a new root. The coproduct on $H_\mathcal{T}$ then satisfies the relation
\begin{equs}
\label{coprodrecur}
	\Delta_{\scriptscriptstyle\mathrm{CK}} (\tau) 
	= \mathbf{1} \otimes \tau + (B_+ \otimes \mathrm{id})\Delta_{\scriptscriptstyle\mathrm{CK}}(\tau_1 \cdots \tau_n).
\end{equs}
\end{remark}

\medskip

We denote by $ \text{char}(H,A) $ the set of characters from the Hopf algebra $H $ into a commutative unital $ \C $-algebra, $ A $. These are linear algebra morphisms forming a group with respect to the convolution product
\begin{equs}
\label{convolution}
	\psi \star \phi := m_A (\psi \otimes \phi ) \Delta
	\qquad
    	H \xrightarrow{\Delta} H \otimes H \xrightarrow{\psi \otimes \phi} 
    	A \otimes A \xrightarrow{m_{A}} A.
\end{equs}

The convolution inverse of a character, $\phi \in \text{char}(H,A) $, is given through composition with the antipode, i.e., $\varphi^{-1} = \varphi \circ \CA $, and the unit for the convolution product is the co-unit $\mathbf{1}^*$ \footnote{For the sake of notational transparency, we'll suppress here the unit-map, $\eta: A \to H$, that should follow the co-unit.}.

It is furthermore assumed that a linear projection, $ Q : A \rightarrow A$, is defined on $A$, which satisfies the (weight $ -1 $) Rota--Baxter identity:
\begin{equs} 
\label{Rota_Baxter}
	Q( f) \bigcdot_{\!\!\scriptscriptstyle A} Q(g)  
	= Q \left( Q(f) \bigcdot_{\!\!\scriptscriptstyle A} \, g
			+ f \bigcdot_{\!\!\scriptscriptstyle A} \, Q(g) 
			 \right) - Q(f \bigcdot_{\!\!\scriptscriptstyle A} \, g)
\end{equs}
for any $ f,g \in A $. Here, $f \bigcdot_{\!\!\scriptscriptstyle A} \, g:=m_A (f \otimes g ) $ denotes the commutative product of $ f $ and $ g $ in the algebra $ A $. The associated projector, $\tilde{Q}:=\id_{A} - Q$, also satisfies identity \eqref{Rota_Baxter}. As a result, $A$ splits into two subalgebras $A_- := Q(A)$ and $ A_+ :=\tilde{Q}(A)$:
\begin{equs}
	A = A_- \oplus A_+.
\end{equs}

\begin{remark} 
One of the main examples is given by the algebra of Laurent series, $A = \C[[t,t^{-1}]$, with finite pole-part. In this context $A_-  =  t^{-1}\C[t^{-1}]$ and $A_+ =  \C[[t]]$, such that $ Q $ keeps only the pole part of a series: 
\begin{equs}
	Q\Big( \sum_{n} a_n t^{n} \Big)  = \sum_{ n< 0} a_n t^{n} \in A_-.
\end{equs}
\end{remark}

\medskip

The Hopf algebraic approach of Connes and Kreimer \cite{CKI} describes the so-called BPHZ renormalisation method in perturbative quantum field theory in terms of a factorisation theorem for -- dimensionally regularised-- Hopf algebra characters. We now recall Connes and Kreimer's so-called algebraic Birkhoff decomposition.

\begin{proposition} \label{classicalBirkhoff}
For every character $ \varphi \in \text{char}(H,A) $, there exist unique algebra morphisms $ \varphi_- : H \rightarrow A_- $ and $ \varphi_+ : H \rightarrow A_+ $ defined in terms of the recursions 
\allowdisplaybreaks
\begin{align}
\label{RBHopfBPHZ}
 	\varphi_-  &= \mathbf{1}^* - Q\big( (\varphi - \mathbf{1}^* ) \star  \varphi_- \big)\\
	\varphi_+ &= \mathbf{1}^* + \tilde{Q}\big( (\varphi - \mathbf{1}^* ) \star  \varphi_- \big)
\end{align}  
and yielding the algebraic Birkhoff factorisation 
\begin{equation}
\label{BPHZfact}
	 \varphi_+ \star \varphi^{-1}_- = \varphi.
\end{equation}
\end{proposition}

\begin{definition}\label{Def:BogoPrep}
The map $\bar \varphi = (\varphi - \mathbf{1}^* ) \star  \varphi_- $ is called Bogoliubov's preparation map. The maps $\varphi_-$ and $\varphi_+$ are called counterterm respectively renormalised character.  
\end{definition}

\begin{remark} 
We note that  $\varphi^{-1}_+ = \varphi_+ \circ \CA $ can be computed recursively
$$
	\varphi^{-1}_+ =  \mathbf{1}^* - \tilde{Q}\big( \varphi^{-1}_+ \star  (\varphi - \mathbf{1}^* ) \big).
$$
Identity \eqref{Rota_Baxter} then implies that $\varphi^{-1}_+  \star \varphi_- = \mathbf{1}^* - \varphi^{-1}_+ \star (\varphi - \mathbf{1}^* ) \star \varphi_- $, from which \eqref{BPHZfact} follows immediately.  
\end{remark}

Observe that $\varphi_{+} =  \varphi \star \varphi_{-} =  m_A \left( \varphi \otimes \varphi_{-} \right) \Delta$ and evaluating on an element $ \tau \in H $ different from $\mathbf{1}$ yields the explicit formulas
\allowdisplaybreaks
\begin{equs}
 \label{Birkhoffrecursion}
 \begin{aligned}
	\varphi_{-}(\tau)   
	&= \mathbf{1}^*(\tau)  - Q\big((\varphi - \mathbf{1}^*)  \star \varphi_- \big)(\tau) 
			    	     = - Q\left(  \bar \varphi (\tau) \right)\\
	\bar \varphi (\tau) 
	&= \varphi(\tau) + \sideset{}{\!^{'}}\sum_{(\tau)} \varphi(\tau')\bigcdot_{\!\!\scriptscriptstyle A} \,  \varphi_{-}(\tau'')\\
	\varphi_{+}(\tau)  
	&= \varphi(\tau) + \sideset{}{\!^{'}}\sum_{(\tau)} \varphi(\tau') \bigcdot_{\!\!\scriptscriptstyle A} \, \varphi_{-}(\tau'') + \varphi_{-}(\tau)  
				    = \tilde{Q}\left(  \bar \varphi (\tau) \right) .
\end{aligned}  
\end{equs}

\begin{remark}  \label{rmk:antipode} 
Note that when $ Q = \id_A $, one recovers the recursive definition for the antipode $\CA$. Indeed, from \eqref{RBHopfBPHZ} we deduce that 
$$
	\varphi_-   = \mathbf{1}^* -  (\varphi - \mathbf{1}^* ) \star \varphi_-  
			=  \frac{1}{\mathbf{1}^* + (\varphi - \mathbf{1}^* )} = \varphi(\mathrm{id}^{-1})
			= \varphi \circ \CA,
$$
which is consistent with the antipode being the convolution inverse of the identity map, implying the antipode recursions (thanks to the connectedness of $H$)
\begin{equs}
	\CA \tau 		& = - \tau - \sideset{}{\!^{'}}\sum_{(\tau)} \tau' \CA\tau'' 
				    =  - \tau - \sideset{}{\!^{'}}\sum_{(\tau)} \CA(\tau') \tau''.
\end{equs}
\end{remark}
In the next section, we consider $ \hat H $ being a right-comodule over $ H $. The space $ \hat H $ is a connected graded unital algebra and we denote its product by $ m_{\hat H} $. The coaction 
$$ 
	\hat \Delta : \hat H \rightarrow \hat H \otimes H 
$$ 
is used to build a variant of factorisation \eqref{BPHZfact}. We suppose that we are also given an injection $ \iota :  H \rightarrow \hat H $.

\begin{proposition} \label{Birkhoff_comodule}
For every $ \varphi \in \text{char}(\hat H,A) $, there are unique linear maps $ \varphi_- : H \rightarrow A_- $ and $ \varphi_+ : \hat H \rightarrow A $:
\begin{equs}
 \label{Birkhoffrecursion2}
 \begin{aligned}
	\varphi_{-}  & = \mathbf{1}^* - Q\circ \bar \varphi \circ \iota, 
	\quad 
	 \bar \varphi=m_A\Big((\varphi - \mathbf{1}^* ) \otimes \varphi_- \Big) \hat \Delta\\
		\varphi_{+} & = \varphi \star \varphi_{-} 
					= m_A \left(   \varphi \otimes \varphi_{-} \right) \hat \Delta,
\end{aligned}  
\end{equs}
where the reduced co-module map is such that for every $ \tau \in H $
\begin{equs}
	\hat \Delta\!' \circ \iota(\tau) 
	=  \sideset{}{\!^{'}}\sum_{(\iota(\tau))} \iota(\tau)' \otimes \iota(\tau)'' 
	= \hat \Delta \circ \iota(\tau) - \iota(\tau) \otimes \one - \one \otimes {\tau}
\end{equs}
corresponds to the reduced coaction.
The linear maps $ \varphi_- $ and $ \varphi_+ $ are also algebra morphisms. Moreover, the map $ \varphi_+ \circ \iota : H \rightarrow A  $ takes values in $ A_+ $.
\end{proposition}
\begin{proof}
It follows essentially from the computation in \eqref{Birkhoffrecursion}, which can be verified using the Rota-Baxter relation. The uniqueness of the decomposition follows from $Q$ being idempotent. 
\end{proof}

\begin{remark} 
Whereas Proposition~\ref{classicalBirkhoff} gives a factorisation in the group $ (\text{char}(H,A), \star) $ of characters, Proposition~\ref{Birkhoff_comodule} does not encode a group factorisation, because the product $ \star  $ is derived from a coaction in this case.
\end{remark}

\begin{remark} \label{twisted_antipodeb} 
When $ Q = \id_A$, one recovers the recursive definition for the so-called twisted antipode $ \tilde{\CA} : H \rightarrow \hat H $:
\allowdisplaybreaks
\begin{equs}
	\tilde{\CA} & =-\iota  -m_{\hat H }(\id_{\hat H}\otimes \tilde{\CA})\hat \Delta\!'\circ \iota \\
	\varphi_{-}  & = \mathbf{1}^* - \bar \varphi \circ \iota 
				= \mathbf{1}^* - m_A\Big((\varphi - \mathbf{1}^* ) \otimes \varphi_- \Big) \hat \Delta \circ \iota 
				= \varphi \circ \tilde \CA.
\end{equs}
\end{remark}

\begin{remark} \label{multiplicative} 
Another example of linear maps $ Q $ that satisfy \eqref{Rota_Baxter} are idempotent algebra morphisms:
\begin{equs} \label{mult_Rota}
	Q(f \bigcdot_{\!\!\scriptscriptstyle A} \, g) = Q(f) \bigcdot_{\!\!\scriptscriptstyle A} \, Q(g), 
	\quad 
	Q \circ Q = Q.
\end{equs}
They appear in the negative renormalisation for Regularity Structures \cite{BHZ}, but also in numerical analysis when one wants to perform local error analysis. We refer the reader to \cite{BS} for details.
\end{remark}

\begin{remark} 
The Rota--Baxter map, $ Q $, can be replaced by a Rota--Baxter family of maps, $ (Q_{\alpha})_{\alpha \in \R_+} $, satisfying the identity \cite{EFGBP2008}
\begin{equs}
	Q_{\alpha}(f)\bigcdot_{\!\!\scriptscriptstyle A} \, Q_{\beta} (g) 
	= Q_{\alpha + \beta} \left(  Q_{\alpha}(f) \bigcdot_{\!\!\scriptscriptstyle A} \,  g 
		+  f  \bigcdot_{\!\!\scriptscriptstyle A} \, Q_{\beta}(g)  
		- f \bigcdot_{\!\!\scriptscriptstyle A} \, g \right).
\end{equs}
\end{remark}


\section{Bogoliubov-type recursions on regularity structure trees}
\label{sec::Birkhoff}

\subsection{Decorated trees}

Recall Remark~\ref{rmk:BCK} that $\mathcal{RT}$ referred to the set of non-planar rooted trees. Let $\Lab$ be a finite set containing so-called types. For a given $ d \in \N $ we define the set of decorations $ \mathcal{D} = \Lab \times \N^{d+1} $ and consider the set $ \mathcal{RT}^{\mathcal{D}} $ of $ \mathcal{D} $-decorated rooted trees that we call  RS (Regularity Structures) Trees. Elements of $\mathcal{RT}^{\mathcal{D}} $ are of the form  $T_{\Labe}^{\Labn} =  (T,\Labn,\Labe) $ where $T$ is a non-planar rooted tree with node and edge sets $N_T$ respectively $E_T$. The maps $ \Labn : N_T \rightarrow \N^{d+1} $ and $\Labe=(\Labe_1,\Labe_2): E_T \rightarrow \mathcal{D}$ are node respectively edge decorations. The tree product $ \bigcdot $ on $ \mathcal{RT}^{\mathcal{D}} $ is defined by 
 \begin{equation} 
 \label{treeproduct}
 	(T,\Labn,\Labe) \bigcdot (\bar T,\bar \Labn,\bar \Labe) 
 	= (T \bigcdot \bar T,\Labn + \bar \Labn, \Labe + \bar\Labe)\;, 
 \end{equation} 
where $ T \bigcdot \bar T $ is the rooted tree obtained by identifying the roots of $ T$ and $\bar T $. The sums $ \Labn + \bar\Labn$ and  $\Labe + \bar \Labe$ mean that decorations are added at the root and extended to the disjoint union by setting them to vanish on the other tree.  In other words, each edge and vertex
of both trees keeps its decoration, except the roots which merge into a new root decorated by
the sum of both corresponding decorations. In this paper, we will use mainly a symbolic notation for these decorated trees. 

\begin{enumerate}
 \item An edge decorated by  $ (\Labhom,p) \in \mathcal{D} $  is denoted by $ \CI_{(\Labhom,p)} $. The symbol $  \CI_{(\Labhom,p)} $ is also viewed as  the operation that grafts a tree onto a new root via a new edge with edge decoration $ (\Labhom,p) $.
 The new root at hand is
decorated with 0. In other words, $ \CI_{(\Labhom,p)} $ sends trees into planted trees.
 \item A factor $ X^k $ encodes a single node  $ \bullet^{k} $ decorated by $ k \in \N^{d+1} $. We write $ X_i $, $ i \in \lbrace 1,\ldots,d+1 \rbrace $, to denote $ X^{e_i} $, where the $\{e_1, \ldots, e_{d+1}\}$ form the canonical basis of $ \N^{d+1} $. The element $ X^0 $ is identified with $ \one $.
 
 \item In the following we will employ a drastically simplified notation for decorated trees by writing: $\widehat{\tau}=T_{\Labe}^{\Labn} \in \mathcal{RT}^{\mathcal{D}}$. 
 \end{enumerate}
 
Recall from Remark \ref{rmk:BCK} the significance of the $B_+$-operation -- in the definition of the Connes--Kreimer coproduct on rooted trees in $\mathcal{RT}$. We use an analog representation for RS trees $\widehat{\tau} \in \mathcal{RT}^{\mathcal{D}} $ using $ \CI_{(\Labhom,p)}$
 \[
	\widehat{\tau} =  X^{k_0} \prod_{i=1}^{n} \CI_{(\Labhom_i,p_i)}(\widehat{\tau}_i) ,
 \]
 where the $\widehat{\tau}_i $ belong to $  \mathcal{RT}^{\mathcal{D}} $ and the product $ \prod_i^n $ is the tree product. The factor $ X^{k_0} $ expresses the fact that the root
of the resulting tree is decorated by $k_0$. The main difference with the $B_+$-operation is that the edges connecting to the new root carry different decorations. Using symbolic notation, one can reformulate the tree product \eqref{treeproduct} as
 \begin{equs}
	\left( X^{k_0} \prod_{i} \CI_{(\Labhom_i,p_i)}(\widehat{\tau}_i)  \right) 
	\left( X^{ k'_0} \prod_{j} \CI_{(\Labhom_j', p_j')}(\widehat{{\tau}}_j') \right) 
	= X^{k_0 + k_0'} \prod_{i,j} \CI_{(\Labhom_i,p_i)}(\hat \tau_i)   \CI_{(\Labhom_j', p_j')}(\widehat{{\tau}}_j').
\end{equs}
The space of $\mathcal{D} $-decorated trees is denoted $ \mathcal{T}^{\mathcal{D}} = \langle \mathcal{RT}^{\mathcal{D}} \rangle $  where $\langle \bigcdot  \rangle$ denotes the $ \R$-linear span. Endowed with the tree product it becomes a commutative algebra. We now associate numbers to decorated trees, depending on the decorations. Further below, it will become clear that they have a transparent interpretation in the context of SPDEs. Let us fix a scaling $ \s \in \N^{d+1} $ and the associate $ |\bigcdot|_{\s} : \Lab \rightarrow \R $. We extend the latter to $ k \in \N^{d+1} $ by $ |k|_{\s}  :=\sum_{i=1}^{d+1} \s_i k_i $. The degree of a decorated rooted tree $ T^{\Labn}_{\Labe} $  is defined by
\begin{equs}
	| T^{\Labn}_{\Labe} |_{\s} 
	=   \sum_{v \in N_{T}}  |\Labn(v)  |_{\s}+ \sum_{e \in E_{T}} | \big(\Labe_1(e) -  \Labe_2(e)\big)   |_{\s},
\end{equs}
where $ \Labe = (\Labe_1,\Labe_2) $. 
 Using this degree, we define the set $ \mathcal{RT}^{\CD}_{+} $ which is included in $ \mathcal{RT}^{\CD}$ by 
\begin{equs} \label{Tplus}
	 \mathcal{RT}^{\CD}_{+} 
 	:=   \Big\lbrace X^{k_0} \prod_{i=1}^{n} \CI_{(\Labhom_i,p_i)}(\widehat{\tau}_i), 
	| \CI_{(\Labhom_i,p_i)}(\widehat{\tau}_i) |_{\s} > 0, \widehat{\tau}_i \in \mathcal{RT}^{\mathcal{D}}, \, k_0 \in \N^{d+1} \Big\rbrace . 
\end{equs}
This definition means that all the branches outgoing from the root must be of positive degree.
We denote by $ \CT_+^{\mathcal{D}} $ the space $ \langle \mathcal{RT}_+^{\mathcal{D}}  \rangle $ and call it the positive part. The corresponding projector $\pi_+$ maps $ \CT^{\mathcal{D}} $ to $ \CT_+^{\mathcal{D}} $. In the following we denote by $ \mathcal{M} $ the tree product on $ \CT^{\mathcal{D}} $.


\subsection{Comodule-Hopf algebra structures}
\label{ssect:comodHopf}

We want to endow the previously introduced algebra on decorated trees with a coproduct which is similar to the Butcher--Connes--Kreimer coproduct \cite{CK} on rooted trees \eqref{coprodrecur}. However, primitiveness of many elements is lost due to the particular nature of the decorations of trees in our setting. We will provide a recursive definition of a coproduct, denoted $\Deltap$, similar to \eqref{coprodrecur} which suffices for formulating the main result, i.e., an algebraic Birkhoff-type factorisation. The map $\Deltap$ is recursively defined on the space of RS trees $\CT^{\CD}$:
\allowdisplaybreaks
\begin{align} 
\begin{split}
	\Deltap \one & = \one \otimes \one, 
	\quad 
	\Deltap X_i  = X_i \otimes \one  + \one \otimes X_i,  \label{coproduct}\\
	\Deltap \CI_{(\Labhom,p)}(\widehat{\tau})&  = 
	\one \otimes \CI_{(\Labhom,p)}(\widehat{\tau}) 
	+  \left( \CI_{(\Labhom,p)} \otimes \id \right) \Deltap \widehat{\tau}
	+\!\! \sum_{\ell \in \N^{d+1} \atop \ell \neq 0}\!\! \frac{X^{\ell}}{\ell !} \otimes \CI_{(\Labhom,p + \ell)}(\widehat{\tau}). 
\end{split}
\end{align}
Recall that $ \one^* : \CT^{\mathcal{D}} \rightarrow \R $ refers to the counit. From the map $ \Deltap $, one can construct a coproduct and a coaction that we denote differently:

\begin{itemize}
\item $ \Deltap :  \CT^{\CD} \rightarrow \CT^{\CD} \otimes \CT^{\CD} $. Here, a specific bigrading is required. This is to make the
above infinite sum well-defined. See \cite[Section 2.3]{BHZ} for more details. A possible choice would be 
\[ 
	(|T_{\Labe}^{\Labn}|_{bi}) = ( |\Labe_2|_{\s}, |N_T \setminus \lbrace \varrho_T \rbrace| + |E_T| ), 
\] 
where  $ |\Labe_2|_{\s} = \sum_{e \in E_T} | \Labe_2(e) |_{\s} $, $ \varrho_T $ is the root vertex of $ T $, $ |N_T| $ and $ |E_T| $ are the numbers of nodes and edges for the tree $T$. This map will have the interpretation of performing an infinite subtraction. Its recursive description appears in \cite[Proposition 4.16]{BHZ}.

\item  $ \hat{ \Delta}^{\!+} = (\id \otimes \pi_+) \Deltap  :  \CT^{\CD} \rightarrow \CT^{\CD} \otimes \CT^{\CD}_{+} $. Here, no bigrading is required, as the sum in \eqref{coproduct} is finite. It should be understood as finite subtraction, with its length determined by the degree of the branch outgoing from the root vertex.

\item $ \bar{ \Delta}^{\!+} = (\pi_+ \otimes \pi_+) \Deltap  :  \CT^{\CD}_{+} \rightarrow \CT^{\CD}_+ \otimes \CT^{\CD}_{+},$ as before no bigrading is required. We put an extra assumption on the trunk by maintaining the positive degree of the branches outgoing from the root. This is a rather strong constraint because the degree of the trunk is lower than that of the original tree (branches of positive degree have been removed).
\end{itemize}

\begin{remark}
In \cite{BS}, a similar coproduct as \eqref{coproduct} is used. The main difference relies on the projection $ \pi_+ $. Indeed, for a numerical scheme the length of the Taylor expansion depends on the order of the scheme, whereas in our context it depends on the regularity of the distribution we would like to re-center.
\end{remark}

\smallskip

Following \cite[Proposition 3.23]{BHZ}, we are able to put a Hopf algebra structure on the different sets of decorated trees: 

\begin{proposition}
We have the following properties

\begin{itemize}
\item Considering the space $\CT^{\mathcal{D}}$ of decorated trees, there exists an algebra morphism $ \CA_+ : \CT^{\mathcal{D}} \rightarrow \CT^{\mathcal{D}}  $  such that  $ H^{\mathcal{D}}=(\CT^{\mathcal{D}}, \mathcal{M}, \Deltap, \one, \one^*, \CA_+  ) $ is a Hopf algebra.

\item Considering the positive part $\CT^{\mathcal{D}}_+ \subset \CT^{\mathcal{D}}$, there exists an algebra morphism $ \bar{\CA}_+ : \CT_+^{\mathcal{D}} \rightarrow \CT_+^{\mathcal{D}}  $  so that  $ H^{\mathcal{D}}_+=(\CT_+^{\mathcal{D}}, \mathcal{M}, \bar{ \Delta}^{\!+}, \one, \one^*, \bar{\CA}_+  ) $ is a Hopf algebra.

\item The map  $  \hat{ \Delta}^{\!+} \colon  \CT^{\CD} \rightarrow \CT^{\CD} \otimes \CT^{\CD}_{+} $ is a coaction satisfying:
\begin{equs}
	\left(  \hat{ \Delta}^{\!+} \otimes \id  \right) \hat{ \Delta}^{\!+} 
	= \left(  \id \otimes  \bar{ \Delta}^{\!+} \right) \hat{ \Delta}^{\!+}
\end{equs}
and turns $ \CT^{\mathcal{D}} $ into a right-comodule for $ \CT^{\mathcal{D}}_+ $. 
\end{itemize}
\end{proposition}

\begin{remark}
The antipode for $ H^{\mathcal{D}} $ is described in terms of a recursive formula given in \cite[Proposition 4.18]{BHZ} by 
\begin{equs}
\begin{split}
	\CA_+ X_i &= - X_i,\\
	\CA_+ \CI_{(\Labhom, p)}(\widehat{\tau}) 
	&= - \sum_{\ell \in \N^{d+1}} \frac{(-X)^{\ell}}{\ell!}  \mathcal{M} \left( \CI_{(\Labhom, p + \ell)}  \otimes \CA_+ \right) \Deltap\widehat{\tau}.
\end{split}
\end{equs}

The infinite sum is made rigorous by use of a specific bigrading. For the antipode of $ H^{\CD}_+$, we get a similar definition \cite[Proposition 6.2]{BHZ}: 

\begin{align} 
\begin{split}
	\bar{\CA}_+ X_i & = - X_i,  \\ 
	\bar{\CA}_+  \CI_{(\Labhom, p)}(\widehat{\tau}) 
	& = - \sum_{\ell \in \N^{d+1}} \frac{(-X)^{\ell}}{\ell!}  \mathcal{M} \left( \pi_+ \CI_{(\Labhom, p + \ell)}  \otimes \bar{\CA}_+ \right) \hat{ \Delta}^{\!+}\widehat{\tau}. \label{plusantipode}
\end{split}
\end{align}
However, the difference is the introduction of the  projector $ \pi_+ $ as well as the coaction $\hat{ \Delta}^{\!+}$, which make the sum finite.
\end{remark}

In fact, the map which is of interest to us, is the one where the projector $ \pi_+$ is not used in \eqref{plusantipode}. This map will be denoted $ \tilde{\CA}_+ : \CT_+^{\mathcal{D}} \rightarrow \CT^{\mathcal{D}} $ and is given in \cite[Proposition 6.3]{BHZ} by:
\begin{align} 
\begin{split}
	\tilde{\CA}_+ X_i & = - X_i, \quad  \\ 
	\tilde{\CA}_+ \CI_{(\Labhom, p)}(\widehat{\tau}) 
	& = - \sum_{|\ell|_{\s} \leq | \CI_{(\Labhom,p)}(\widehat{\tau}) |_{\s}} \frac{(-X)^{\ell}}{\ell!}  \mathcal{M} \left(  \CI_{(\Labhom, p + \ell)}  \otimes \tilde{\CA}_+ \right) \hat{ \Delta}^{\!+}\widehat{\tau}. \label{twistedantipode}
\end{split}
\end{align}
It is called twisted antipode and plays a major role in describing the local behaviour of solutions of singular SPDEs.
In \eqref{twistedantipode}, the projection $ \pi_+ $ is replaced by a global one which is based on the degree of the decorated tree on which we apply the twisted antipode.
 In this work we will reinterpret this map as a Bogoliubov-type recursion. The aim is to bridge the gap between the renormalisation procedure developed for singular SPDEs and Connes--Kreimer's formulation of the BPHZ renormalisation procedure in perturbative quantum field theory in terms of an algebraic Birkhoff factorisation on the level of regularised Feynman rules seen as a Hopf algebra character. In the next section, we introduce a modified reduced coproduct together with a Rota--Baxter map essential for Bogoliubov-type recursions.

\smallskip 

Before stating the definition of a modified reduced coproduct, we notice that the decorated tree $\bullet^n = X^n$ is not primitive with respect to the coproduct \eqref{coproduct}. Indeed
\begin{equs}
	\Deltap X^n = X^n \otimes \one + \one \otimes X^n 
				+ \sum_{k \in \N^{d+1} \atop k \neq 0,n} \binom{n}{k} X^{k} \otimes X^{n-k} ,
\end{equs}
where the binomial coefficient $\binom{n}{k}:=\prod_{i=1}^{d+1}\binom{n_i}{k_i}$ and $\binom{n_i}{k_i}$ is zero when $k_i$ is bigger than $n_i$. We refer the reader to \cite[Section 2]{BHZ} for details. For a tree of the form $ \CI_{(\Labhom,p + \ell)}(\widehat\tau) $, called planted tree, the main difference with the Butcher--Connes--Kreimer coproduct \eqref{coprodrecur} is that one goes to the next orders by adding derivatives and polynomial decorations in \eqref{coproduct}. These extra terms are given by:
\begin{equs}
	\sum_{\ell \in \N^{d+1} \atop \ell \neq 0} \frac{X^{\ell}}{\ell !} 
	\otimes \CI_{(\Labhom,p + \ell)}(\widehat\tau). 
\end{equs} 
A natural choice is to also remove those terms from a potential definition of what we will call modified reduced coproduct. We need to generalise this procedure to arbitrary trees, i.e., products of planted trees with a non-zero decoration at the root. The basic idea is to remove any polynomial part on the righthand side of the modified reduced coproduct.

\begin{definition}
The modified reduced coproduct map $ \Delta^{\! +}_{\scriptscriptstyle{\mathrm{red}}} $ is given on $ X^{k_0} $ by:
\begin{equs}
	 \Delta^{\! +}_{\scriptscriptstyle{\mathrm{red}}} X^{k_0} = 0.
\end{equs}
 Then for any rooted tree $\widehat\tau =  X^{k_0} \prod_{i=1}^{n} \CI_{(\Labhom_i,p_i)}(\widehat\tau_i) \in  \CT^{\mathcal{D}} $ we set:
\begin{equs}
\label{modredcoprod}
\begin{aligned}
	\Delta^{\! +}_{\scriptscriptstyle{\mathrm{red}}} \widehat\tau 
	& =  \hat \Delta^{\! +} \widehat\tau  - \widehat\tau \otimes \one\\ 
	& -\!\!\!  \sum_{ \ell_1,\ldots,\ell_n \atop \ell_i,k \in \N^{d+1}}\!\! \frac{1}{ \bar \ell!}\binom{k_0}{k} 
	X^{k+ \sum_i \ell_i}  \otimes X^{k_0-k} \prod_{i=1}^{n} \pi_+ \CI_{(\Labhom_i,p_i + \ell_i)}(\widehat\tau_i)
\end{aligned}
\end{equs}
where $ \bar \ell ! = \prod_i \ell_i ! $ and for $ k_0 = 0 $, the sum over $ k $ contains only the term $  k = 0 $ by convention.
\end{definition}

\begin{remark} Sweedler's notation is used for the modified reduced coproduct \eqref{modredcoprod}:
\begin{equs} \label{Sweedler_reduced}
	\Delta^{\! +}_{\scriptscriptstyle{\mathrm{red}}} \widehat\tau 
	= \sideset{}{^+}\sum_{(\widehat\tau)} \widehat{\tau}' \otimes \widehat{\tau}''.
\end{equs} 
\end{remark}

\begin{remark}
The definition of the modified reduced coproduct implies the primitiveness of $ \bullet^n $, i.e., $ \Delta^{\! +}_{\scriptscriptstyle{\mathrm{red}}} X^{n} = 0 $. Moreover, one gets the following recursion:
\begin{equs} 
\label{recursivereduced}
	\Delta^{\! +}_{\scriptscriptstyle{\mathrm{red}}}\CI_{(\Labhom,p)}(\widehat\tau)  
	= \left( \CI_{(\Labhom,p)} \otimes \gamma \right)  \hat \Delta^{\! +} \widehat\tau,
\end{equs}
where $ \gamma = \id - \one^{*} $ is the augmentation projector, which is zero on the empty tree $\one$ and the identity otherwise.
\end{remark}


\subsection{Bogoliubov's recursion}

After having introduced the modified reduced coproduct \eqref{modredcoprod}, we need to consider the space of characters, which will be used to iterate the Bogoliubov-type recursion. These are linear maps from $H^{\mathcal{D}} $ to the specific target space of functions, $ \CH = \mathcal{C}^{\infty}(\R^{d+1},\R) $, respecting the tree product $\mathcal{M}$. 
The recursion we want to set up will produce a decomposition of a specific character into a product of two characters, and will depend on a parameter $ x \in \R^{d+1} $. The latter is used to fix the corresponding splitting of the target space
\begin{equs}
	\CH = \CH_{x}^{+} \oplus \CH_x^{-},
\end{equs} 
where $\CH_{x}^{+}  $ contains the  non-polynomial functions vanishing at $ x $ and $ \CH_x^{-} $ consists of polynomial functions whose coefficients are functions of $ x $. Indeed, for any $ f \in \CH $ one has the straightforward splitting $ f = f - f(x) + f(x) $. The next definition presents the key map lying at the origin of the Bogoliubov-type recursion. It has to be understood as a Taylor jet (of order $\alpha$).

\begin{definition}
We set for $\alpha\in\R_+$, $x, y\in\R^{d+1}$ and $f\in\CH$
\begin{equs}
\label{op}
	\mathrm{T}_{\alpha,x,y}f \eqdef \sum_{\ell \in \N^{d+1} \atop |\ell|_{\s} < \alpha } \frac{(y-x)^{\ell}}{\ell !} (D^{\ell}f)(x).
\end{equs}
\end{definition}
 From \eqref{op}, we immediately see that $ \mathrm{T}_{\alpha,x,y}f \in \CH_x^- $ when $ f \in \mathcal{H} $. The next lemma provides some key properties of the Taylor jet \eqref{op} which will be used in the sequel.

\begin{lemma}\label{lem:RBb}
The operators $ \mathrm{T}_{\bigcdot,x,\bigcdot} $ defined in $ \eqref{op} $ satisfy for every $\alpha,\beta \in\R_+$ and functions $f, g\in\CC^\infty(\R^{d+1},\R) $ the following Rota--Baxter-type identity
\begin{equs} \label{RB}
	 (\mathrm{T}_{\alpha,x,\bigcdot} f)(\mathrm{T}_{\beta,x,\bigcdot} g) 
	 = \mathrm{T}_{\alpha + \beta,x,\bigcdot} [ (\mathrm{T}_{\alpha,x,\bigcdot} f) g 
	 + f (\mathrm{T}_{\beta,x,\bigcdot} g) -f g ].
 \end{equs}
 For a fixed but arbitrary $\bar{x} \in \R^{d+1}$ we have
 \begin{equs} \label{identaylor}
	 \mathrm{T}_{\alpha,x,y} f 
	 = \sum_{|\ell|_{\s} < \alpha  } \frac{(y-\bar x)^{\ell}}{\ell !} 
	 \mathrm{T}_{\alpha- |\ell|_{\s},x,\bar x} [D^{\ell} f].
 \end{equs}
\end{lemma}
\begin{proof}
The first identity \eqref{RB} is well-known in the literature and corresponds to the notion of family of Rota--Baxter maps \cite{EFGBP2008}. We give a proof of the second one \eqref{identaylor} which is in fact essential for the sequel. One has
\allowdisplaybreaks 
\begin{equs}
	\sum_{|\ell|_{\s} < \alpha  } \frac{(y-\bar x)^{\ell}}{\ell !} \mathrm{T}_{\alpha- |\ell|_{\s},x,\bar x} [D^{\ell} f] 
	& = \sum_{|\ell|_{\s} < \alpha  } \frac{(y-\bar x)^{\ell}}{\ell !}  \sum_{|k|_{\s} < \alpha - |\ell|_{\s} } \frac{(\bar x- x)^{k}}{k !}  D^{\ell+k} f(x)
	\\ & = \sum_{|\ell|_{\s} < \alpha  } \frac{(y-\bar x + \bar x - x)^{\ell}}{\ell !}   D^{\ell} f(x)  
	\\ & = \mathrm{T}_{\alpha,x,y} f.
\end{equs}
\end{proof}

\begin{remark} \label{extRB}
Identity~\eqref{RB} is also true for $ \alpha, \beta \in \R_- $ by setting $ \mathrm{T}_{\alpha,x,\bigcdot} f = 0 $ whenever $ \alpha \leq 0 $.
\end{remark}

We consider now linear maps from $H^{\CD}$ to $\CH$ parametrised by finite sets of elements in $\R^{d+1}$. Of particular interest is the family of algebra morphisms 
\begin{equs}
	\Phi = (\varphi_{x_1,\ldots,x_n})_{ x_1,\ldots,x_n \in \R^{d+1}},
	\quad  
	\varphi_{x_1,\ldots,x_n} : H^{\CD} \rightarrow \CH.
\end{equs}
 For given  $ x_1,...,x_n $, $ \varphi_{x_1,...,x_n} $ is a function that we will also denote in the sequel by:
\begin{equs}
\varphi_{x_1,...,x_n} : y \mapsto \varphi_{x_1,...,x_n,y} \in \R.
\end{equs}
The variables $ x_1,..., x_n $ correspond to some quantities fixed within the definition of $ \varphi $. 
When $ n=1 $, it is interpreted as the fact that some origin has been fixed and that characters $\varphi_{x_1} : H^{\CD} \rightarrow \CH$ contain partially a re-centering. Indeed, they are defined on $ \bullet^k $ and the natural evaluation should give the polynomial function associated to $X^k$. Therefore, our polynomial functions will be re-centered around this parameter.

\medskip

We consider now the linear maps 
\begin{equs}
\varphi_{\bar x} \colon H^{\CD} \rightarrow \CH, \quad \varphi^{-}_{x, \bar x } : H^{\CD}_+ \rightarrow \CH_x^{-}, \quad \bar \varphi_{x,\bar x} : H^{\CD} \rightarrow \CH, \quad \varphi^{+}_{x, \bar x}: H^{\CD} \rightarrow \CH 
\end{equs} in $ \Phi$ and follow Proposition \ref{Birkhoff_comodule} in setting up a Bogoliubov-type recursion for the counter term map. With our notations, one has for every $ y \in \R^{d+1} $:
\begin{equs}
\varphi_{\bar x,y}, \, \varphi^{-}_{x, \bar x ,y}, \,  \varphi_{x,\bar x,y}, \, \varphi^{+}_{x, \bar x,y} \in \R. 
\end{equs}
Here we see any map $\varphi^{-}_{x, \bar x } : H^{\CD}_+ \rightarrow \CH_x^{-}$ as element of $ \Phi $ because these maps admit a natural extension to $ H^{\CD} $ by setting them to be zero outside $ H^{\CD}_+ $. It is assumed that $ \varphi_{\bar x} $ is a character parametrised by ${\bar x} \in \R^{d+1} $. The map $ \bar \varphi_{x,\bar x}  $ plays the role of Bogoliubov's preparation map obtained from the modified reduced coproduct $ \Delta^{\! +}_{\scriptscriptstyle{\mathrm{red}}} $. Upon applying the Taylor jet operator, Bogoliubov's preparation map gives the counter term character, $ \varphi^{-}_{x, \bar x }$. Eventually, the renormalised character $ \varphi^{+}_{x, \bar x} $ follows from a convolution between $ \varphi^{-}_{x, \bar x} $ and $\varphi_{\bar x }$ using the coaction $  \hat \Delta^{\! +} $. We will skip the injection from $H^{\CD}_+$ into $H^{\CD}$ for notational clarity. 

\begin{definition} \label{Bogoliubov}
Let $\varphi_{\bar x} \in \Phi $, ${\bar x} \in \R^{d+1} $ be a character. We set up the following Bogoliubov-type recursion with respect to the points $ x, y \in \R^{d+1} $, $ \widehat\tau \in \CR \CT^{\CD} $: 
\allowdisplaybreaks
\begin{equs}  \label{Bogoliubovrecursion}
\begin{aligned}
	\bar \varphi_{x,\bar x,y} (\widehat\tau) 
	& = \varphi_{\bar x, y }(\widehat\tau) 
		+ \sideset{}{^+}\sum_{(\widehat\tau)} \varphi_{\bar x, y }(\widehat\tau') \varphi^{-}_{ x, \bar x, \bar x }(\widehat\tau'') \\
	\varphi^{-}_{x, \bar x, y }(\widehat\tau) 
	& = - \mathrm{T}_{|\widehat\tau|_{\s},x,y } \left(  \bar \varphi_{x,\bar x,\bigcdot} (\widehat\tau) \right) \\
	\varphi^{+}_{x, \bar x,y}
	& = \varphi_{\bar x,y} \star \varphi^{-}_{{ x, \bar x, \bar x } }
	= \left(   \varphi_{\bar x, y }\otimes {\varphi^{-}_{ x, \bar x, \bar x }} \right) \hat \Delta^{\! +}.
\end{aligned}    
\end{equs}
\end{definition}

\begin{remark} 
Using the recursive formulation of the modified reduced coproduct \eqref{recursivereduced}, one gets the following identity for $ \bar \varphi_{x,\bar x,\bigcdot} $
\begin{equs} \label{RCidentity}
	\bar \varphi_{x,\bar x,y} (\CI_{(\Labhom,p)}(\widehat\tau))  
	= \varphi_{\bar x, y }(\CI_{(\Labhom,p)}(\widehat\tau)) 
		+\left( \varphi_{{ \bar x, y }} \CI_{(\Labhom,p)} \otimes \varphi^{-}_{{ x, \bar x, \bar x }}  \right) 
		\hat \Delta^{\! +} \widehat\tau.
\end{equs}
\end{remark}

\medskip 

 The central point is to verify the following theorem that gather various important properties:
\begin{theorem} \label{main_propoerties} Within the set up of Definition~\ref{Bogoliubov}, one has
\begin{enumerate}
\item The counter term map $ \varphi^{-}_{x,\bar x} $ is an algebra morphism from $ H^{\CD}_{+} $ into $ \CH_x^{-} $.

\item The renormalised map  $ \varphi^{+}_{x, \bar x} $ is an algebra morphism from $ H^{\CD} $ into $ \CH $, which sends  trees from $ \CT^{\CD}_+ $ into $ \CH_x^{+} $.

\item One can prove that the maps $ \varphi^{+}_{x,\bar x}   $,  and $ \bar{\varphi}_{x,\bar x} $ are invariant in their second
subscript. The map $ \varphi^{-}_{x,\bar x} $ has also this invariance but only on decorated trees with zero node decoration.
\end{enumerate}
\end{theorem}
\begin{proof} The first two points are described in Theorem~\ref{main theorem}. The third point is covered in Theorem~\ref{invariant bar x}.
\end{proof}

\begin{assumption} \label{assumpt1}
We assume that the family of characters $ (\varphi_{\bar x})_{ \bar x \in \R^{d+1}} $ satisfies:
\begin{equs}
	 \varphi_{\bar x,y} (X_i)   = y_i - \bar x_i, 
 	\quad  
 	\varphi_{\bar x, \bigcdot}(\CI_{(\Labhom,\ell)}(\widehat\tau)) 
	= D^{\ell} \varphi_{\bar x, \bigcdot}(\CI_{(\Labhom,0)}(\widehat\tau)).
\end{equs}
\end{assumption}

The first identity in Assumption~\ref{assumpt1} corresponds to interpreting the point $ \bar x $ as a re-centering of polynomials.
The second identity shows that adding decoration to an edge amounts in fact to taking derivatives. This identity is crucial and combined with \eqref{identaylor} allows us to prove the character property of the counter term $ \varphi^{-}_{x,\bar x} $. For the rest of this section, we consider a family of characters $ (\varphi_{\bar x})_{ \bar x \in \R^{d+1}} $ satisfying this assumption.

From \eqref{RCidentity} and Assumption~\ref{assumpt1}, we have for $ \CI_{(\Labhom,k+ \ell)}(\widehat\tau) \in \mathcal{T}^{\mathcal{D}}_+ $  that
\begin{equs} \label{phideri}
	 \bar \varphi_{x,\bar x,\bigcdot}(\CI_{(\Labhom,k+ \ell)}(\widehat\tau)) 
	 = D^{\ell} \bar \varphi_{x,\bar x,\bigcdot}(\CI_{(\Labhom,k)}(\widehat\tau)).
\end{equs}

Assumption~\ref{assumpt1} gives also the behaviour of Bogoliubov's recursion on the polynomials. From the primitiveness, $ \Delta^{\! +}_{\scriptscriptstyle{\mathrm{red}}}  X^k = 0$, one gets:
\begin{equs}
\bar \varphi_{x,\bar x,y}(X^k) = \varphi_{\bar x,y}(X^k) = (y-\bar{x})^k.
\end{equs}
Then, 
\begin{equs}
	\varphi^{-}_{ x, \bar x, y}(X^k) 
	&  =  - \mathrm{T}_{|k|_{\s},x,y } \left(  \bar \varphi_{x,\bar x,\bigcdot} (X^k) \right)
 	    = - \mathrm{T}_{|k|_{\s},x,y } \left( (\bigcdot - \bar x)^k \right)\\ 
	& = - \sum_{|\ell|_{\s} < |k|_{\s}} \binom{k}{\ell} (y-x)^{\ell} (x - \bar x)^{k-\ell}.
\end{equs}
For $ y = \bar x $, we then have
\begin{equs}
	\varphi^{-}_{ x, \bar x, \bar x}(X^k) 
	&  = - \sum_{|\ell|_{\s} < |k|_{\s}} \binom{k}{\ell} (\bar x-x)^{\ell} (x - \bar x)^{k-\ell}
	\\ & =  - \sum_{|\ell|_{\s} \leq |k|_{\s}} \binom{k}{\ell} (\bar x-x)^{\ell} (x - \bar x)^{k-\ell} + (\bar x - x)^{k}
	\\ & = (\bar x - x)^{k}.
\end{equs}
At the end,
\begin{equs}
	\varphi^{+}_{x, \bar x, y }(X^k) 
	& = \left(   \varphi_{\bar x, y }\otimes {\varphi^{-}_{ x, \bar x, \bar x }} \right) \hat \Delta^{\! +} X^k
	\\&= \sum_{|\ell|_{\s} \leq |k|_{\s}} \binom{k}{\ell} \varphi_{\bar x, y}(X^{\ell}) \varphi^{-}_{ x, \bar x, \bar x}(X^{k-\ell})
 	\\ & = \sum_{|\ell|_{\s} \leq |k|_{\s}} \binom{k}{\ell} (y-\bar x)^{\ell} (\bar x - x)^{k-\ell}
	 \\ & = (y-x)^k.
\end{equs}

\begin{theorem} \label{main theorem}
The map $ \varphi^{-}_{x, \bar x} $ is an algebra morphism from $ H^{\CD}_{+} $ into $ \CH_x^{-} $ and $ \varphi^{+}_{x, \bar x} $ is an algebra morphism from $ H^{\CD} $ into $ \CH $. 
\end{theorem}

\begin{proof}
We proceed by induction on the number of edges and the decoration at the root. We consider trees $ \widehat\tau_1,  \widehat\tau_2 \in \mathcal{RT}_{+}^{\mathcal{D}} $. According to our convention they both admit symbolic representations of the form:
\begin{equs}
	\widehat\tau_1 = X^{n}\prod_i \CI_{(\Labhom_i, k_i)}(\widehat\tau_{1i}), 
	\quad 
	\widehat\tau_2 = X^{\bar n} \prod_j \CI_{(\bar{\Labhom}_j, \bar{k}_j)}(\widehat\tau_{2j}).
\end{equs}
We expand the recursive expression of the counter term $ \varphi^{-}_{x,\bar x}  $ giving for $ y \in \R^{d+1} $:
\begin{equs}
	 \varphi^{-}_{x,\bar x, y}  (\widehat\tau_1\widehat\tau_2) 
	& = - \mathrm{T}_{|\widehat\tau_1\widehat\tau_2|_{\s},x,y}\left( \varphi_{\bar x, \bigcdot}(\widehat\tau_1\widehat\tau_2) 
	+ \sideset{}{^+}\sum_{(\widehat\tau_1\widehat\tau_2)}  
	\varphi_{\bar x, \bigcdot}\left(\lbrace \widehat\tau_1\widehat\tau_2 \rbrace'\right) 
	\varphi^{-}_{x,\bar x, \bar x} \left(\lbrace \widehat\tau_1\widehat\tau_2 \rbrace''\right) \right) \\
	& = - \mathrm{T}_{|\widehat\tau_1\widehat\tau_2|_{\s},x,y} \left(  \left\lbrace \varphi_{\bar x,\bigcdot}(\widehat\tau_1) 
	+ \sideset{}{^+}\sum_{(\widehat\tau_1)} \varphi_{\bar x, \bigcdot}(\widehat\tau_1') 
	\varphi^{-}_{x,\bar x, \bar x}\left(\widehat\tau_1'' \right) \right\rbrace \right. \\ 
	& \left.\left\lbrace \varphi_{\bar x, \bigcdot}(\widehat\tau_2) 
		+ \sideset{}{^+}\sum_{(\widehat\tau_2)} \varphi_{\bar x, \bigcdot}(\widehat\tau_2') 
		\varphi^{-}_{x,\bar x, \bar x}\left(\widehat\tau_2'' \right) \right\rbrace \right.\\ 
	& + \sum_{\ell_1,\ldots, \ell_{d+1} \atop |k|_{\s} < |n|_{\s}  } \binom{n}{k}  \varphi_{\bar x, \bigcdot} \left(\frac{X^{\sum_i \ell_i + k}}{ \ell !} \right) 		\varphi^{-}_{x,\bar x, \bar x} \left( X^{n-k} \prod_j \CI_{(\Labhom_j, k_j + \ell_j)}(\widehat\tau_{1j}) \right) \\ 
	& \left\lbrace \varphi_{\bar x,\bigcdot}(\widehat\tau_2) 
	+  \sideset{}{^+}\sum_{(\widehat\tau_2)}  \varphi_{\bar x, \bigcdot} \left(\widehat\tau_2' \right) 
		\varphi^{-}_{x,\bar x, \bar x}\left(\widehat\tau_2'' \right) \right\rbrace  \\
	& \left. + \sum_{\ell_1,\ldots, \ell_{d+1} \atop |k' |_{\s} < |\bar n|_{\s}} \binom{\bar n}{k'}  
	\varphi_{\bar x, \bigcdot} \left(\frac{X^{\sum_i \ell_i + k'}}{\ell !} \right) 
	\varphi^{-}_{x,\bar x, \bar x}\left( X^{\bar n  - k'} \prod_j \CI_{(\bar \Labhom_j, \bar k_j + \ell_j)}(\widehat\tau_{2j}) \right) \right.\\ 
	& \left.  \left\lbrace \varphi_{\bar x, \bigcdot}(\widehat\tau_1) 
	+ \sideset{}{^+}\sum_{ (\widehat\tau_1)} \varphi_{\bar x, \bigcdot} \left(\widehat\tau_1' \right) \varphi^{-}_{x,\bar x, \bar x}\left( \widehat\tau_1'' \right) \right\rbrace \right), 
\end{equs}
where $  \ell ! = \prod_i \ell_i ! $ and we have used several times the Sweedler's notation for the reduced coaction $\Delta^{\! +}_{\scriptscriptstyle{\mathrm{red}}}$ see \eqref{Sweedler_reduced}.  The computation above illustrates the fact that the reduced coaction $\Delta^{\! +}_{\scriptscriptstyle{\mathrm{red}}}$ is not multiplicative and one has to add more terms involving monomials of the form $ X^{\sum_i \ell_i + k} $ and $ X^{\sum_i \ell_i + k'} $.
By applying the induction hypothesis on each $  X^{n-k} \prod_i \CI_{ (\Labhom_i, k_i + \ell_i)}(\widehat\tau_{1i}) $, we obtain:
\begin{equs}
	\varphi^{-}_{x,\bar x, \bar x}\left( X^{n-k} \prod_j\CI_{( \Labhom_j,  k_j + \ell_j)}( \widehat\tau_{1j}) \right) 
	= \varphi^{-}_{x,\bar x, \bar x}\left(  X^{n-k} \right) 
	\prod_j\varphi^{-}_{x,\bar x, \bar x}\left(  \CI_{( \Labhom_j,  k_j + \ell_j)}( \widehat\tau_{1j}) \right) .
\end{equs}

We use identity~\eqref{identaylor} to get for each  $ \CI_{(\Labhom_i,k_i)}( \widehat\tau_{1i})   $ with $ \alpha_i := | \CI_{(\Labhom_i,k_i)}( \widehat\tau_{1i}) |_{\s} $:
\begin{equs}
	\varphi^{-}_{x, \bar x, y}\left(  \CI_{(\Labhom_i, k_i)}(\widehat\tau_{1i}) \right)  
	& = - \mathrm{T}_{\alpha_i,x, y } \left(  \bar \varphi_{x,\bar x,\bigcdot} (\CI_{(\Labhom_i,k_i)}(\widehat\tau_{1i})) \right)\\ 
	& = - \sum_{|\ell_j |_{\s} < \alpha_i} \frac{(y-\bar x)^{\ell_j}}{\ell_j !} 
	\mathrm{T}_{\alpha_i- \ell_j,x,\bar x} \big[D^{\ell_j}  \bar \varphi_{x,\bar x,\bigcdot} (\CI_{(\Labhom_i, k_i)}(\widehat\tau_{1i}))\big ]\\ 
	& = - \sum_{|\ell_j |_{\s} < \alpha_i } \frac{(y-\bar x)^{\ell_j}}{\ell_j !}
	 \mathrm{T}_{\alpha_i- \ell_j,x,\bar x} 
	 \big[  \bar \varphi_{x,\bar x,\bigcdot} (\CI_{(\Labhom_i, k_i+ \ell_j)}(\widehat\tau_{1i}))\big ]\\ 
	 & = \sum_{|\ell_j |_{\s} < \alpha_i} \varphi_{\bar x,\bigcdot} \left(\frac{X^{\ell_j}}{\ell_j !} \right) 
	 \varphi^{-}_{x, \bar x, \bar x}\left(  \CI_{(\Labhom_i, k_i + \ell_j)}(\widehat\tau_{1i}) \right). 
\end{equs}
On the other hand we get:
\begin{equs}
	\varphi^{-}_{x, \bar x, \bigcdot}\left( X^n \right) 
	= \sum_{ |k|_{\s} < | n |_{\s} } \binom{ n}{k}  \varphi_{\bar x, \bigcdot} \left( X^{ k} \right) 
	\varphi^{-}_{x, \bar x,\bar x}\left( X^{n-k} \right).  
\end{equs}

We conclude by applying identity~\eqref{RB}
\begin{equs}
	\varphi^{-}_{x,\bar x, y}  (\widehat\tau_1\widehat\tau_2) 
	& = - \mathrm{T}_{|\widehat\tau_1\widehat\tau_2|_{\s},x,y} 
	\left(\bar \varphi_{ x, \bar x, \bigcdot}(\widehat\tau_1)  \bar \varphi_{x, \bar x, \bigcdot}(\widehat\tau_2)  
	- \mathrm{T}_{|\widehat\tau_1|_{\s},x,\bigcdot} \left( \bar \varphi_{x,\bar x,\bigcdot}(\widehat\tau_1) \right) 
	\bar \varphi_{x,\bar x, \bigcdot}(\widehat\tau_2) \right. \\ 
	& \qquad - \left. \bar \varphi_{x,\bar x, \bigcdot}(\widehat\tau_1) 
	\mathrm{T}_{|\widehat\tau_2|_{\s},x,\bigcdot} \left( \bar \varphi_{x,\bar x, \bigcdot}(\widehat\tau_2) \right) \right) \\ 
	& =  \mathrm{T}_{|\widehat\tau_1|_{\s},x,y} \left( \bar \varphi_{x, \bar x, \bigcdot}(\widehat\tau_1) \right)  
	\mathrm{T}_{|\widehat\tau_2|_{\s},x,y} \left( \bar \varphi_{x, \bar x,\bigcdot}(\widehat\tau_2) \right) \\
	& = \varphi^{-}_{x, \bar x, y}(\widehat\tau_1) \varphi^{-}_{x, \bar x,y}(\widehat\tau_2).
\end{equs}
 The fact that $ \varphi^{+}_{x,\bar x, y} $ is an algebra morphism follows from its definition together with $ \varphi_{\bar x} $ and $ \varphi^{-}_{x,\bar x, y} $ being algebra morphisms. 
\end{proof}

\begin{remark} 
From the proof above, one can see that for a tree $ \widehat\tau $ with a negative sub-branch we get $ \varphi^{-}_{x, \bar x, \bigcdot}(\widehat\tau) = 0 $ by using the extension mentioned in remark~\ref{extRB}.
\end{remark}

In the next proposition, we connect the twisted antipode, $\tilde \CA_+ \colon  \CT^{\CD}_+ \to  \CT^{\CD} $,  given in \eqref{twistedantipode} with the character $ \varphi^{-}_{x,\bar x,\bar x}  \colon H^{\CD}_{+} \to \R$ constructed through the Bogoliubov-type recursion.

\begin{proposition} \label{twisted_antipode_f} 
Let $\varphi_{\bar x, x} \colon H^{\CD} \to \R$ be a character. One has the following identity between characters from  $H^{\CD}_{+}$ to $\R$
\begin{equs}
	\varphi^{-}_{x,\bar x, \bar x} = \varphi_{\bar x, x} \tilde \CA_+.
\end{equs}
\end{proposition}

\begin{proof} 
We proceed by induction on trees. Recall that Assumption~\ref{assumpt1} is in place. By multiplicativity, we just consider a planted tree $ \CI_{(\Labhom,k)}(\widehat\tau) $ and set $ \alpha = |\CI_{(\Labhom,k)}(\widehat\tau)  |_{\s} $:
\begin{equs}
	 \varphi_{\bar x, x} \left( \tilde \CA_+ \CI_{(\Labhom,p)}(\widehat\tau) \right) 
	 = - \sum_{|\ell|_{\s} \leq \alpha } \frac{\varphi_{\bar x, x}(-X)^{\ell}}{\ell!}   
	 \left(\varphi_{\bar x, x} \CI_{(\Labhom, p + \ell)}  \otimes 
	 \varphi_{\bar x, x}\tilde{\CA}_+ \right) \hat{ \Delta}^{\!+} \widehat\tau .
\end{equs}
We use the inductive hypothesis and $ \varphi_{\bar x, \bigcdot} (\CI_{(\Labhom, p + \ell)}(\widehat \tau)) = D^{\ell} \varphi_{\bar x,\bigcdot} (\CI_{(\Labhom, p)}(\widehat \tau)) $ to get
\begin{equs}
	 \varphi_{\bar x, x} \left( \tilde \CA_+ \CI_{(\Labhom,p)}( \widehat \tau) \right) 
	 & = - \sum_{|\ell|_{\s} \leq \alpha } \frac{(\bar x - x)^{\ell}}{\ell!}  \left(  \left( D^{\ell}   \varphi_{\bar x, \bigcdot}\big( \CI_{(\Labhom, p)} (\bigcdot)\big)\right)(x) \otimes \varphi^{-}_{x,\bar x, \bar x} \right) \hat{ \Delta}^{\!+} \widehat \tau \\
 	& = - \left( \mathrm{T}_{\alpha,x,\bar x } \left( \varphi_{\bar x, \bigcdot}\big(\CI_{(\Labhom, p)} (\bigcdot)\big) \right)  
		\otimes \varphi^{-}_{x,\bar x, \bar x} \right) \hat{ \Delta}^{\!+} \widehat\tau \\
 	& = - \mathrm{T}_{\alpha,x,\bar x } \left( \varphi_{\bar x, \bigcdot}(\CI_{(\Labhom, p)} (\widehat\tau)) \right)  
 	- \left( \mathrm{T}_{\alpha,x,\bar x } \left( \varphi_{\bar x, \bigcdot} \right)  \otimes \varphi^{-}_{x,\bar x, \bar x} \right)
	 	\Delta^{\! +}_{\scriptscriptstyle{\mathrm{red}}}\CI_{(\Labhom, p)}(\widehat\tau)\\
 	& = - \mathrm{T}_{\alpha,x,\bar x } \left(  \bar \varphi_{x,\bar x,\bigcdot} (\CI_{(\Labhom,p)}(\widehat\tau)) \right) \\ 
	& = \varphi^{-}_{x, \bar x, \bar x }(\CI_{(\Labhom,p)}(\widehat \tau)).
\end{equs}
\end{proof}

Given a tree $ \widehat\tau = X^{n}\prod_i \CI_{(\Labhom_i, k_i)}(\widehat{\tau}_i), $ with $ \alpha_i = | \CI_{(\Labhom_i, k_i)}(\widehat\tau_i) |_{\s} $, we define for every $ y \in \R^{d+1} $:
\begin{equs}
	\big(\bar \varphi_{x,\bar x,y}  - \mathrm{T}_{|\bigcdot|_{\s},x,y } 
	& \left(  \bar \varphi_{x,\bar x}  \right)\big)(\widehat\tau)  
	:= \left( \bar \varphi_{x,\bar x,y}(X^n)  
	- \mathrm{T}_{|n|_{\s},x,y } \left(  \bar \varphi_{x,\bar x,\bigcdot} (X^n) \right) \right) \\ 
	& \prod_i  \Big(\bar \varphi_{x,\bar x,y}(\CI_{(\Labhom_i, k_i)}(\widehat\tau_i))  - \mathrm{T}_{\alpha_i,x,y } 
	\left(  \bar \varphi_{x,\bar x,\bigcdot} (\CI_{(\Labhom_i, k_i)}(\widehat\tau_i)) \right)\Big)
\end{equs}

\begin{proposition} \label{explicit phi plus} 
One has for every $ \widehat\tau \in \CT^{\CD} $:
\begin{equs}
	\varphi_{\bar x,x,y}^{+}(\widehat\tau) 
	= \left( \bar \varphi_{x,\bar x,y}  
		- \mathrm{T}_{|\bigcdot|_{\s},x,y } \left(  \bar \varphi_{x,\bar x} \right) \right)(\widehat\tau). 
\end{equs}
Therefore $ \varphi_{\bar x,x,\bigcdot}^{+}(\widehat\tau) $ belongs to $ \CH_x^{+} $.
\end{proposition}

\begin{proof} 
By multiplicativity, we just need to check this property for $ X_i $ and trees of the form $\CI_{(\Labhom,k)} (\widehat\tau)$. First one can check that
\begin{equs}
	 \bar \varphi_{x,\bar x,y} (X_i) - \mathrm{T}_{1,x,y } \left(  \bar \varphi_{x,\bar x,\bigcdot} (X_i) \right) 
 	& = \varphi_{\bar x,y} (X_i) -   \varphi_{\bar x,x} (X_i)   \\ 
	& = (y_i - \bar x_i) - ( x_i - \bar x_i)  
	   = y_i - x_i.
\end{equs}
From $ X_i $ being primitive follows that:
\begin{equs}
	 \varphi^{+}_{x, \bar x ,y}(X_i) 
	 =\left(   \varphi_{\bar x, y }\otimes \varphi^{-}_{ x, \bar x, \bar x } \right) \hat \Delta^{\! +} X_i = y_i -x_i.
\end{equs}
For $ \widehat\tau = \CI_{(\Labhom,k)} (\widehat\tau_1) $, we get
\begin{equs}
	\lefteqn{(  \varphi_{\bar x, y } 
	 \otimes \varphi^{-}_{ x, \bar x, \bar x } ) \hat \Delta^{\! +} \widehat\tau
	 = (  \varphi_{\bar x, y } 
	 \otimes \varphi^{-}_{ x, \bar x, \bar x } ) \hat \Delta^{\! +}\CI_{(\Labhom,k)} (\widehat\tau_1)} \\ 
	& = \left(  \varphi_{\bar x, y } \CI_{(\Labhom,k)} \otimes \varphi^{-}_{ x, \bar x, \bar x } \right) \hat \Delta^{\! +} \widehat\tau_1 
	+ \sum_{|\ell|_{\s} < | \widehat\tau |_{\s}} \left(  \varphi_{\bar x, y } \frac{X^{\ell}}{\ell !} \otimes \varphi^{-}_{ x, \bar x, \bar x } \CI_{(\Labhom,k+ \ell)}(\widehat\tau_1)\right)\\ 
	& = \bar \varphi_{x,\bar x,y} (\widehat\tau) - \sum_{|\ell|_{\s} < |\widehat\tau|_{\s}} \left(  \frac{(y - \bar x)^{\ell}}{\ell !}  
	\mathrm{T}_{|\widehat\tau|_{\s} - | \ell |_{\s},x,\bar x} \big[ \bar \varphi_{ x, \bar x, \bigcdot } \CI_{(\Labhom,k+ \ell)}(\widehat\tau_1)\big]\right) \\ 
	& =  \bar \varphi_{x,\bar x,y} (\widehat \tau) - \mathrm{T}_{|\hat \tau|_{\s},x,y } \left(  \bar \varphi_{x,\bar x,\bigcdot} (\widehat\tau) \right),
\end{equs}
where we have used the identities \eqref{identaylor} and \eqref{phideri}.
\end{proof}

In order to get some $ \bar x $-invariant properties, we need to be more precise in the choice of the characters which is the aim of the next assumption.

\begin{assumption} \label{assumpt2}
We assume that the family of characters $ (\varphi_{\bar x})_{ \bar x \in \R^{d+1}} $ satisfies:
\begin{equ}[e:Pim]
  \left\{ \begin{aligned}
  &   \varphi_{\bar x,y} (X_i)   = y_i - \bar x_i,   \\
  &   \varphi_{\bar x,y}\big( \CI_{(\Labhom,k)} (\widehat\tau)\big) = 
  \int_{\R^{d+1}} D^{k} K_{\Labhom}(y-z) \varphi_{\bar x,z} (\widehat\tau) dz
  \end{aligned} \right.
\end{equ}
where $ (K_{\Labhom})_{\Labhom \in \Lab} $ is a family of smooth and compactly supported kernels.
\end{assumption}
Note that Assumption~\ref{assumpt2} implies  Assumption~\ref{assumpt1}. Indeed, the kernels $ K_{\mathfrak{t}} $ being smooth and compactly supported, one can exchange derivatives and integrals. A good example to have in mind is when the $ K_{\mathfrak{t}} $ are in $ \mathcal{C}^{\infty}(\R^{d+1},\R) $ and compactly supported.

\begin{theorem} \label{invariant bar x} 
Under the Assumption~\ref{assumpt2}, one gets
\begin{enumerate} 
\item The renormalised character map $ \varphi^{+}_{x, \bar x } $ is invariant in its second subscript: $\varphi^{+}_{x, 0}:= \varphi^{+}_{x, \bar x }$. 
\item  Bogoliubov's preparation map $ \bar \varphi_{x, \bar x }   $ is invariant in its second
subscript on planted trees and:
\begin{equs}
\bar \varphi_{x, \bar x, y }\big(\CI_{(\Labhom,k)} (\widehat\tau)\big) 
	= \big( D^{k} K_{\Labhom} * \varphi^{+}_{x, \bar{x}}(\widehat\tau) \big) (y).  
\end{equs}
\item The counter term map $ \varphi^{-}_{x, \bar x }   $ is invariant in its second
subscript on trees which have a zero node decoration at the root.
\end{enumerate}
\end{theorem}

\begin{proof} We first prove the invariance of the renormalised character map $ \varphi^{+}_{x,\bar x} $ in its second subscript. By the morphism property, we need to check it on $ \one $,  $ X_i $ and $ \CI_{(\Labhom,k)} (\widehat\tau) $. One has:
\begin{equs}
\varphi^{+}_{x, \bar x ,y}(\one) = 1, \quad 	\varphi^{+}_{x, \bar x ,y}(X_i)  = y_i - x_i.
\end{equs}
Then for $ \CI_{(\Labhom,k)} (\widehat\tau) $, one gets
\begin{equs}
 	\bar \varphi_{x, \bar x ,\bigcdot}(\CI_{(\Labhom,k)} (\widehat\tau)) 
 	& =   \left( \varphi_{\bar x, \bigcdot} \CI_{(\Labhom,k)} \otimes \varphi^{-}_{ x, \bar x, \bar x } \right) \hat \Delta^{\! +} \widehat\tau \\  
	& = \left(  D^{k} K_{\Labhom}  * \varphi_{\bar x, \bigcdot }\otimes \varphi^{-}_{ x, \bar x, \bar x } \right) \hat \Delta^{\! +} \widehat\tau \\
 	& =  \big( D^{k} K_{\Labhom} * \varphi^{+}_{x, \bar x}(\widehat\tau) \big) (\bigcdot).
\end{equs} 
We conclude by applying an induction hypothesis on $ \varphi^{+}_{x, \bar x, \bigcdot}(\widehat\tau) $ and by using proposition~\ref{explicit phi plus}. We also proved the formula for $ \bar \varphi_{x, \bar x ,y} $ and its invariance in the second subscript on planted trees. For $\varphi^{-}_{x, \bar x }$, we use the character property and we consider only planted trees. By the definition of $ \varphi^{-}_{x, \bar x }$ which involves $\bar \varphi_{x, \bar x}   $, we conclude from the invariance of $\bar \varphi_{x, \bar x}   $ on its second subscript.
\end{proof}
 \begin{remark} Let us mention that the map $ \varphi^{-}_{x, \bar{x}} $ is not invariant in its second subscript on planted trees with a non zero node decoration at the root. Indeed, this produces a factor of the form $ \varphi^-_{x,\bar{x}}(X^k) $ which is not invariant.
\end{remark}

\section{Applications to singular SPDEs}
\label{sec::4}
 In this section, we present  algebraic Birkhoff type factorisations connected to Proposition~\ref{Birkhoff_comodule} and its extension with Theorem~\ref{main_propoerties}. In each case, we will specify the coaction and the target space $ A $ respectively $ A_- $ and $ A_+ $, which are different depending on the application. The first application is the recentering character for Regularity Structures also called Positive Renormalisation and first introduced in \cite{reg}. Its construction using a twisted antipode has been made precise in \cite{BHZ}. We revisit its construction in light of Theorem~\ref{main_propoerties}.
The second application is the negative renormalisation for SPDEs, which is reproducing the BPHZ algorithm by incorporating all the Taylor expansions at the level of the algebra. It was introduced in \cite{BHZ} and is also described by a twisted antipode. We reinterpret this construction with  Proposition~\ref{Birkhoff_comodule} and a Rota-Baxter map satisfying the identity \eqref{mult_Rota}.

\subsection{Positive Renomalisation} \label{sec::singular}

The formalism of decorated trees has been developed originally for singular Stochastic Partial Differential Equations (SPDEs). In this context, a decorated tree is a combinatorial representation of an iterated integral obtained from a perturbative expansion performed using the mild formulation of the equation.
Indeed, considering an SPDE of the form:
\begin{equs} \label{SPDEs}
	\partial_t u - \Delta u = F(u,\nabla u, \xi), \quad  (t,x) \in \R_+ \times \R^{d},
\end{equs}
where $ \xi $ is a space-time noise, the mild formulation of \eqref{SPDEs} is given by:
\begin{equs} \label{mild}
	u = K * \left(  F(u, \nabla u, \xi) \right),
\end{equs}
where $ K $ is the heat kernel and $ * $ denotes space-time convolution. We suppose for simplicity that $u(0,\bigcdot) = 0$. Then, if $ F $ is a polynomial, one iterates this equation on $ u  $ as well as $ \nabla u $. The noise $ \xi $ is replaced by a mollified noise $ \xi_{\eps} $. If $ F $ is not a polynomial, the iteration is performed after its Taylor expansion, which makes appear polynomials in the iterated integrals.

More generally, we suppose given a finite set $ \Lab = \Lab_+ \sqcup \Lab_- $, a collection of kernels
$ (K_{\Labhom})_{\Labhom \in \Lab_+} $ which corresponds to the kernels appearing in the mild formulation and a collection of smooth noises $ (\xi_{\Labhom})_{\Labhom \in \Lab_-} $. The scaling $ \s \in \N^{d+1} $ is fixed to be parabolic, i.e., $ \s = (2,1,\ldots,1) $. Furthermore, it is supposed that $ \eqref{SPDEs} $ is locally subcritical in the sense that one can associate a subcritical and normal complete rule $ R $ to $ \eqref{SPDEs} $. Then, one can construct a subspace $  \bar\CT^{\CD}  = \langle  B_{\circ}\rangle \subset \CT^{\CD} $, which is stable under the coproduct on decorated trees and generated by $ R $. The set $ B_{\circ} \subset  \CR \CT^{\CD}$ are the decorated that conform the rule $ R $. In other words, they are generated by iterating the mild formulation \eqref{mild}.  We shall refrain from giving further details and refer the reader to  \cite[Section 5]{BHZ}, where those rules have been detailed. The decorated trees in $\bar\CT^{\CD}$ associated to the equation allow to describe the iterated integrals coming from its perturbative expansion. 

The essential point is that $  \bar\CT^{\CD}$ is a subset of $\CT^{\CD} $ but not necessarily a subalgebra.  Indeed, the rules reflect the distributional products of the equation, which imply that trees cannot be multiplied using the tree product: yielding products that we will not be able to renormalise.

 However, the (co)algebraic results given in Section~\ref{sec::Birkhoff} remain valid when one replaces $ \CT^{\CD} $ by $ \bar\CT^{\CD}  $.  This is possible because $ \bar\CT^{\CD}  $ is compatible with the coproduct and one can use the product in $ \CT^{\CD} $ to decompose elements of $ \bar\CT^{\CD}  $ into smaller components. Note that decorated trees are still abbreviated by $\widehat\tau \in  \bar\CT^{\CD}  $.

Then, the iterated integrals are given by a family of characters, $ \PPi^{(\bar x)} : \bar\CT^{\CD}  \rightarrow \CH $, indexed by $ \bar x \in \R^{d+1} $ and recursively defined by
\begin{equ}[e:Pim]
  \left\{ \begin{aligned}
  & (\PPi^{(\bar x)} \one)(y) = 1 ,  
  \qquad
  (\PPi^{(\bar x)} X_i)(y)  = y_i - \bar x_i,\\ 
  & (\PPi^{(\bar x)} \CI_{(\Labhom,k)} (\widehat\tau))(y) = 
  \int_{\R^{d+1}} D^{k} K_{\Labhom}(y-z) (\PPi^{(\bar x)} \widehat\tau)(z) dz, \quad \Labhom \in \Lab_+ \\
  & (\PPi^{(\bar x)} \CI_{(\Labhom,k)} (\one))(y) = 
 D^{k} \xi_{\Labhom}(y), \quad \Labhom \in \Lab_-.
  \end{aligned} \right.
\end{equ}
We add as an extra assumption that edges of type $ \Labhom \in \Lab_- $ can only appear as terminal edges. We also assume that branches of the form $\CI_{(\Labhom,k)}(X^\ell)$
with $ \ell \neq 0 $ and $ \Labhom \in \Lab_- $  cannot appear. The character 
\begin{equs} \label{PPi}
	\PPi := \PPi^{(0)} 
\end{equs} 
has been introduced in \cite{reg} where it was identified as the main character for the construction of the model in \cite{BHZ}. However, we would like to consider instead $ \PPi^{(\bar x)} $, where a priori recentering of the polynomials around $ \bar x \in \R^{d+1} $ is allowed, and see how the main objects used in the theory of SPDEs behave toward this change.

As presented in the previous section, we are given a scaling $ \s $ together with a map $ |\bigcdot|_{\s} : \Lab \rightarrow \R $. The latter depends on the analytical assumptions put on the kernels and noises. It is expected that $|\bigcdot|_{\s} : \Lab_+ \rightarrow \R_+  $ corresponds to Schauder estimates for improvement of regularity and that  $|\bigcdot|_{\s} : \Lab_- \rightarrow \R_-  $ encodes the singularity of the noises when the mollification is removed.
We consider the space $ \bar\CT^{\CD}_+ $ as not being a subspace of $ \bar\CT^{\CD} $, but as a subspace of  
\begin{equs} \label{TplusSPDE}
	 \bar\CT^{\CD}_+ \subset \hat \CT^{\CD}_{+} 
	 =   \Big\{ X^{k} \prod_{i=1}^{n} 
	 \hat \CJ_{(\Labhom_i,p_i)}( \widehat\tau_i),   \widehat\tau_i \in \bar\CT^{\CD}  \Big\}.
\end{equs}
Here the symbol $ \hat \CJ_{(\Labhom,p)}  $ instead of $  \CI_{(\Labhom,p)} $ is used for the edges outgoing of the root. We denote by $  \mathfrak{i}_+$  the injection from $\hat \CT^{\CD}_+$ to $ \CT^{\CD}$. Elements of   $ \bar \CT^{\CD}_+ $ are defined by applying the projection map $ \pi_+ $ which is now considered from $ \hat \CT^{\CD}_+ $ to $ \bar \CT^{\CD}_+ $. The symbol $ \CJ_{(\Labhom,p)}  $ will be a short hand notation for $ \pi_+ \circ \hat \CJ_{(\Labhom,p)}  $. The map $ \PPi^{(\bar x)}  $ can also be defined on $ \hat \CT^{\CD}_+ $ by setting:
\begin{equs}
	\PPi^{(\bar x)} \hat \CJ_{(\Labhom,p)}(\widehat\tau) := \PPi^{(\bar x)} \CI_{(\Labhom,p)}(\widehat\tau). 
\end{equs}

The main reason for marking this difference is that $\bar\CT^{\CD}$, in general, is not an algebra. Indeed, the construction of $\bar\CT^{\CD} $ is constrained by the distributional product appearing on the right-hand side of \eqref{SPDEs}. On the other hand, $ \bar\CT^{\CD}_+ $ is always an algebra. 

One of the main achievements of Hairer's theory of regularity structures is to provide a Taylor expansion of the solution $ u $ of \eqref{SPDEs}:
\begin{equs}
	u(y) = u(x) + \sum_{\widehat\tau \in \bar\CT^{\CD} _{\tiny{\text{eq}}}} \Upsilon[\widehat\tau](x) (\Pi_x \widehat\tau)(y) + R(x,y),
\end{equs}
where $ \bar\CT^{\CD} _{\tiny{\text{eq}}} \subset \bar\CT^{\CD} $ contains trees generated by the perturbative expansion and $ \Upsilon[\widehat\tau](x) $ are coefficients which may depend on $ u $ and its derivatives $ \nabla u $. The map $ \Pi_x $ is deduced from $ \PPi^{(0)} = \PPi $ in Definition~\ref{def:modelmap} below.  Then, one can define a Regularity Structure $ (\bar\CT^{\CD} , G) $  where
\begin{itemize}
	\item $ \bar\CT^{\CD}  = \bigoplus_{\alpha \in A} \bar\CT^{\CD}_{\alpha} $ is a graded space with $ A \subset \R  $ bounded from below and locally finite.  For any element $ \widehat\tau \in \bar\CT^{\CD}_{\alpha} $, one has $ |\widehat\tau|_{\s}  = \alpha $. When $ \widehat\tau \in \bar\CT^{\CD} $,  $ \Vert  \widehat\tau \Vert_{\alpha} $ denotes the norm of its component in $\bar\CT^{\CD}_{\alpha}$.
	\item $ G $ is a structure group of continuous
linear operators acting on  $\bar\CT^{\CD}$ such that, for every $ \Gamma \in G $, every $ \alpha \in A $ and $ \widehat\tau \in  \bar\CT^{\CD}_{\alpha} $, one has
\begin{equs}
	\Gamma \widehat\tau - \widehat\tau \in \bigoplus_{\beta < \alpha}  \bar\CT^{\CD} _{\beta} .
\end{equs}
Elements in $ G $ are maps of the form $ \Gamma_{g} $ where $ g :  \bar\CT^{\CD}_+  \rightarrow \R $ is a character and
\begin{equs}
	\Gamma_g  = \left( \id \otimes g \right) \hat{ \Delta}^{\!+}.
\end{equs}
\end{itemize}
One of the main definitions in \cite{reg,BHZ} makes precise the structures behind the construction of the map $ \Pi_x $.

\begin{definition}\label{def:modelmap}
Given the linear map $\PPi\colon  \bar\CT^{\CD}  \to \CH$, we define for all $z,\bar z\in\R^{d+1}$ 
\begin{itemize}
\item a linear map $\Pi_z\colon  \bar\CT^{\CD}  \to \CH$ and a character
$f_z :  \bar\CT^{\CD}_+ \rightarrow \CH$ by
\begin{equ}[e:defModel1]
	\Pi_z = \bigl(\PPi \otimes f_z\bigr) \hat \Delta^{\! +}  \;,
	\qquad 
	f_z = \left(\PPi \tilde\CA_+ \bigcdot \right)(z)\;,
\end{equ}
where $\tilde \CA_+$ is the positive twisted antipode given in \eqref{e:pseudoant} below.

\item a linear map $\Gamma_{\!z\bar z}\colon \bar\CT^{\CD} \to \bar\CT^{\CD}$ and a character $\gamma_{\!z\bar z} : \bar\CT^{\CD}_+  \rightarrow \CH$
\begin{equ}[e:defModel2]
	\Gamma_{\!z\bar z} = \left( \id \otimes \gamma_{\!z\bar z} \right) \hat \Delta^{\! +} , 
	\quad
 	\gamma_{\!z\bar z} =  \bigl( f_z\CA_{+} \otimes f_{\bar z}\bigr)\bar \Delta^{\! +}  \;.
\end{equ}
\end{itemize} 
\end{definition}

Then under certain favourable conditions on the map $ \PPi $ which correspond to the definition given in \eqref{e:Pim}, $ (\Pi, \Gamma) $ is an admissible model satisfying the following:
\begin{itemize}
\item Algebraic properties:
\begin{equs}  \label{alge_identity}
	\Gamma_{xx} 
	= \id, \quad \Gamma_{xy} \circ \Gamma_{yz} 
	= \Gamma_{xz}, \quad \Pi_y = \Pi_x \circ \Gamma_{xy}
\end{equs}
\item Analytical bounds: For every compact set $\K \subset \R^{d+1}$, we assume the existence of a constant $C_{\ell,\K}$ such that the bounds
\begin{equ}[e:boundPi]
	|(\Pi_x \widehat\tau)(y)| \le C_{\ell,\K} \|\widehat\tau\|_\ell \, \|x-y\|_\s^\ell, 
	\qquad
	\|\Gamma_{xy} \widehat\tau\|_m \le C_{\ell,\K} \|\widehat\tau\|_\ell \, \|x-y\|_\s^{\ell-m}\;, 
\end{equ}
hold uniformly over all $x,y \in \K$, all $m\in A$ with $m < \ell$ and all $\widehat\tau \in \bar\CT^{\CD}$ such that $ |\widehat\tau|_{\s} = \ell $.
\end{itemize}

 From \cite[Prop. 6.3]{BHZ} it follows that there exists a unique algebra morphism $\tilde\CA_+ \colon \bar\CT^{\CD}_+ \to \hat \CT^{\CD}_+$, called the ``positive twisted antipode'', such that $\tilde\CA_+ X_i = - X_i$ and furthermore for all $\CJ_{(\Labhom,k)}(\widehat\tau)\in \bar \CT^{\CD}_+$ 
\begin{equ}[e:pseudoant]
	\tilde\CA_+ \CJ_{(\Labhom,k)}(\widehat\tau) 
	= -\sum_{|\ell|_{\s} < | \CI_{(\Labhom,k)}(\widehat\tau) |_{\s} }
	{(-X)^\ell \over \ell!}  \hat \CM_+ \bigl( \hat \CJ_{(\Labhom,k+\ell)} 
	\otimes \tilde\CA_+\bigr)  \hat \Delta^{\! +} \widehat\tau\;,
\end{equ}
where $\hat \CM_+ : \hat \CT^{\CD}_+ \otimes \hat \CT^{\CD}_+ \rightarrow \hat \CT^{\CD}_+  $ is the tree product on $ \hat \CT^{\CD}_+ $. Then we consider the character $ \varphi $ defined for every $ \bar x \in \R^{d+1} $ by:
\begin{equs} 
\label{def_character}
	\varphi_{\bar x}(\widehat\tau) = (\PPi^{(\bar x)} \widehat\tau).
\end{equs}

The main characters used for defining the model associated to a SPDE are:
\begin{equs}
	f^{(\bar x)}_x = (\PPi^{(\bar x)} \tilde \CA_+ \bigcdot)(x), 
	\quad 
	\Pi_{x}^{(\bar x)} = \left( \PPi^{(\bar x)} \otimes f_x^{(\bar x)} \right) 
\hat \Delta^{\! +}
\end{equs}
 Then, one can define the re-expanding map $ \Gamma^{(\bar x)} $ as
\begin{equs}
\Gamma^{(\bar x)}_{xy} = \left( \id \otimes \gamma^{(\bar x)}_{xy} \right) \hat \Delta^{\! +}, \quad \gamma^{(\bar x)}_{xy} = \left( f^{(\bar x)}_x (\CA_+ \bigcdot)\otimes f^{(\bar x)}_y \right) \Deltap. 
\end{equs}

\begin{proposition} \label{BModel}
Under the assumption~\eqref{def_character}, one gets:
\begin{equs}
	\Pi_x^{(\bar x)}  = \varphi^{+}_{x, \bar x}, 
	\quad 
	f_{x}^{(\bar x)} = \varphi^{-}_{x, \bar x, \bar x} .
\end{equs}
\end{proposition}

\begin{proof} 
The fact that $ f_{x}^{(\bar x)} = \varphi^{-}_{x, \bar x, \bar x}  $ comes from Proposition~\ref{twisted_antipode_f}. Then, we conclude using Definition~\ref{Bogoliubov}.
\end{proof}

\begin{corollary} 
The model $ (\Pi,\Gamma) $ is invariant under translations of the monomials, in the sense that for every $ \bar x \in \R^{d+1} $:
\begin{equs}
	\Pi_x  = \Pi_x^{(\bar x)}, 
	\quad
	\Gamma_{xy}  = \Gamma_{xy}^{(\bar x)}.
\end{equs}
\end{corollary}

\begin{proof} For $ \Pi_x^{(\bar x)} $, the invariance follows from Theorems~\ref{invariant bar x} and \ref{BModel}. Then for $ \Gamma_{xy} $, we proceed by induction through the formula introduced in \cite[Prop. 3.13]{Br17}:
\begin{equs}
	\Gamma_{xy} \CI_{(\Labhom,k)}(\widehat\tau) 
		= \CI_{(\Labhom,k)}(\Gamma_{xy} \widehat \tau) 
			+ \sum_{|\ell|_{\s} \leq | \CI_{(\Labhom,k)}(\widehat\tau) |_{\s}}
			\frac{(X + (y-x) \one)^{\ell}}{\ell !}
			(\Pi_x \CI_{(\Labhom,k+ \ell)}( \Gamma_{xy} \widehat\tau)) (y).
\end{equs}
\end{proof}

We are also able to recover the recursive formulation proposed as a definition in \cite{reg} for the maps $ \Pi_x $ and  $ f_{x}^{(\bar x)} $. 

\begin{proposition}\label{recursive formulation}
The map $ \Pi_x $ is given for $ \CI_{(\Labhom,k)} (\widehat\tau)  $ with $|\CI_{(\Labhom,k)} (\widehat \tau)|_{\s} = \alpha $ by
\begin{equs} 
\label{formula Pi} 
\begin{aligned}
	\left( \Pi_x \CI_{(\Labhom,k)} (\widehat\tau) \right)(y) 
	& = (D^{k} K_{\Labhom} * \Pi_{x}\widehat\tau)(y) \\ 
	&  - \sum_{|\ell|_{\s} \leq \alpha} \frac{(y-x)^{\ell}}{\ell !} (D^{k+\ell} K_{\Labhom} * \Pi_{x} \widehat\tau)(x)
\end{aligned}
\end{equs}
and the map $ f_x^{(\bar x)} $ is given for $ \CJ_{(\Labhom,k)} (\widehat\tau)  $ with $ \alpha = | \CI_{(\Labhom,k)} (\widehat\tau)|_{\s} > 0 $ by 
\begin{equs} 
\label{formula fx}
	f_x^{(\bar x)} \left(  \CJ_{(\Labhom,k)} (\widehat\tau) \right) 
	&  = - \sum_{| \ell |_{\s} \leq \alpha} \frac{(\bar x  - x)^{\ell}}{\ell !} (D^{k+\ell} K_{\Labhom} * \Pi_{x} \widehat\tau)(x).
\end{equs} 
\end{proposition}

\begin{proof}
Indeed, from Theorem~\ref{invariant bar x}, one has an explicit formula for the preparation map $ \bar \varphi_{x, \bar x } $ given by
\begin{equs}
	\bar \varphi_{x, \bar x }(\CI_{(\Labhom,k)} (\widehat\tau)) 
	=  D^{k} K_{\Labhom} * \varphi^{+}_{x}(\widehat\tau)=  D^{k} K_{\Labhom} * \Pi_{x}(\widehat\tau).  
\end{equs}
Therefore, for $ \CI_{(\Labhom,k)} (\widehat\tau)  $ with $|\CI_{(\Labhom,k)} (\widehat\tau)|_{\s} = \alpha $
\begin{equs}
\left( \Pi_x \CI_{(\Labhom,k)} (\widehat \tau) \right)(y) & = \left( \id - \mathrm{T}_{\alpha,x,\bigcdot } \right) \left(  \bar \varphi_{x,\bar x} ( \CI_{(\Labhom,k)} (\widehat\tau) ) \right)(y) 
\\ & = \left( \id - \mathrm{T}_{\alpha,x,\bigcdot } \right) \left( D^{k} K_{\Labhom} * \Pi_{x}(\widehat\tau) \right) (y)
\\ & = (D^{k} K_{\Labhom} * \Pi_{x}\widehat\tau)(y) - \sum_{|\ell|_{\s} \leq \alpha} \frac{(y-x)^{\ell}}{\ell !} (D^{k+\ell} K_{\Labhom} * \Pi_{x} \widehat\tau)(x).
\end{equs}
Then, for $  \CJ_{(\Labhom,k)} (\widehat\tau)  $ with $ \alpha = | \CI_{(\Labhom,k)} (\widehat\tau)|_{\s} > 0 $, we get:
\begin{equs}
	f_x^{(\bar x)} \left(  \CJ_{(\Labhom,k)} (\widehat \tau) \right) 
	& =  - \mathrm{T}_{\alpha,x,\bar x} \left(  \bar \varphi_{x,\bar x} (  \mathfrak{i}_+(\CJ_{(\Labhom,k)} (\widehat\tau)) ) \right) \\
	& = - \mathrm{T}_{\alpha,x,\bar x} \left(  \bar \varphi_{x,\bar x} ( \CI_{(\Labhom,k)} (\widehat\tau) ) \right)\\
	& = \sum_{| \ell |_{\s} \leq \alpha} \frac{(\bar x  - x)^{\ell}}{\ell !} (D^{k+\ell} K_{\Labhom} * \Pi_{x} \widehat\tau)(x).
\end{equs}
One recovers for $ \bar x  = 0 $ the formula given in \cite[Lemma 6.10]{BHZ}.
\end{proof}

If we take $\bar x = x $, we obtain a simple formula for $ f_{x}^{( \bar x)} $:
\begin{equs}
	f_x^{(x)}( \CJ_{(\Labhom,k)}(\widehat\tau)) 
	= -  \one_{ |  \CI_{(\Labhom,k)}(\widehat\tau) |_{\s} > 0 }   \left( D^{k} K_{\Labhom} * \Pi_{x} \widehat\tau \right)(x).
\end{equs}
 When $ \bar x = 0 $, which corresponds to the case considered in \cite{reg,BHZ}, one can fix $ x=0 $ and obtains 
\begin{equs}
	f_0( \CJ_{(\Labhom,k)}(\widehat\tau)) 
	= -  \one_{ |  \CI_{(\Labhom,k)}(\widehat\tau) |_{\s} > 0 } \left( D^{k} K_{\Labhom} * \Pi_{0} \widehat\tau \right)(0).
\end{equs}
If we consider only this case, some simplifications can be introduced in the algebraic structures. 
 The  Bogoliubov recursion given in \eqref{Bogoliubovrecursion} is the same except that now one can replace the map $ \Deltap $ by changing its value on polynomials:
\begin{equs}
\Deltap X_i = X_i \otimes \one, 
\end{equs}
This will also change the definition of the space $ \bar \CT^{\CD}_+ $ which will no longer contain the polynomials. One gets the following  definition:
\begin{equs} 
	\bar \CT^{\CD}_{+} 
	=   \Big\lbrace  \prod_{i=1}^{n}  \CJ_{(\Labhom_i,p_i)}(\widehat\tau_i), 
	| \CI_{(\Labhom_i,p_i)}(\widehat\tau_i) |_{\s} > 0, \widehat\tau_i \in \bar\CT^{\CD} \Big\rbrace . 
\end{equs}

Then the coaction $  \hat{ \Delta}^{\!+} : \bar\CT^{\CD} \rightarrow \bar\CT^{\CD} \otimes \bar\CT^{\CD}_+ $ is given by:
\begin{equs} \label{coaction_s} 
\begin{aligned}
  	\hat{ \Delta}^{\!+} X_i 
	&  = X_i \otimes \one,   \\
 	\hat{ \Delta}^{\!+} \CI_{(\Labhom,p)}(\widehat\tau)
 	&  =  \left( \CI_{(\Labhom,p)} \otimes \id \right)  \hat{ \Delta}^{\!+} \widehat\tau 
 	+ \sum_{\ell \in \N^{d+1}} \frac{X^{\ell}}{\ell !} \otimes \CJ_{(\Labhom,p + \ell)}(\widehat\tau)
 \end{aligned} 
 \end{equs}
We keep the structure of the deformation in this definition, but now $ X_i $ is not primitive. The main change occurs for the coproduct $ \bar{ \Delta}^{\!+} $ where all the deformation at the edge adjacent to the root is removed. 
 \begin{equs} \label{coproduct_s}
	 \bar{ \Delta}^{\!+} \CJ_{(\Labhom,p)}(\widehat \tau)
	 &  =  \left( \CJ_{(\Labhom,p)} \otimes \id \right)  \hat{ \Delta}^{\!+} \widehat\tau 
	 	+ \one \otimes \CJ_{(\Labhom,p)}(\widehat\tau) .
 \end{equs}
 Moreover, the Hopf algebra $ \bar \CT^{\CD}_+ $ is now connected and one gets
 \begin{equs} \label{antipose_s}
 	\CA_+ X_i &  = - X_i, 
	\quad 
	\CA_+ \CJ_{(\Labhom, p)}(\widehat\tau) 
		= -   \mathcal{M} \left( \CJ_{(\Labhom, p )}  \otimes \CA_+ \right)  \bar{ \Delta}^{\!+} \widehat\tau.
\end{equs}
Even the twisted antipode $ \tilde{\CA}_+ $ is simplified in this context:
\begin{equs} \label{twisted_antipose_s}
	\tilde\CA_+ \CJ_{(\Labhom,k)}(\widehat\tau) = 
 	- \hat \CM_+ \bigl( \hat \CJ_{(\Labhom,k)} \otimes \tilde\CA_+\bigr) \hat{ \Delta}^{\!+} \widehat\tau.
\end{equs}
In certain cases, it will also coincide with the antipode. Indeed, one can have:
\begin{equs}
	\hat{ \Delta}^{\!+} \widehat\tau 
	= \sum_{(\widehat \tau)} \widehat\tau' \otimes \widehat\tau'', 
	\quad  
	| \CI_{(\Labhom,k)}(\widehat\tau')|_{\s} \geq 0.
\end{equs}
In fact, this case is almost the general case. Indeed, it happens for branched rough paths, generalised KPZ equations and on every equation when the space $ \bar\CT^{\CD}_{\text{eq}} $ is a positive sector, that is, a subspace invariant by $ G $ and where the minimum degree is $0$. For a negative sector, under certain conditions one can remove the negative part with a generalised Da Prato--Debussche trick see \cite[Section 5]{BCCH} and work on a positive sector.  We can reformulate \eqref{Bogoliubovrecursion} in the following proposition:
\begin{proposition} \label{simplification_1}
 The Definition~\ref{Bogoliubov} becomes for $ x=0 $
\begin{equs} \label{BogoliubovrecursionV2} 
\begin{aligned}
	\bar \varphi (\widehat\tau)(y) 
	& = \varphi(\widehat\tau)(y) 
	+ \sideset{}{\!^{+}}\sum_{(\widehat\tau)} \varphi(\widehat\tau')(y) \varphi^{-}(\widehat\tau'') (y) \\
	\varphi^{-}(\widehat\tau)(y) 
	& = - \mathrm{T}_{|\widehat \tau|_{\s},0,y } \left(  \bar \varphi (\widehat\tau) \right) \\
	\varphi^{+}(\bigcdot)(y) 
	& = \varphi(\bigcdot)(y) \star \varphi^{-} (\bigcdot)(0) 
		= \left(   \varphi(\bigcdot)(y) \otimes \varphi^{-}(\bigcdot)(0) \right) \hat \Delta^{\! +}   
\end{aligned} 
\end{equs}
where
\begin{equs}
	\varphi = \PPi, 
	\quad 
	\varphi^{+} = \Pi_0 , 
	\quad 
	\varphi^-(\bigcdot)(x) = f_{x}.
\end{equs}
\end{proposition}
\begin{proof}
The simplification comes from the fact that terms of the form 
\begin{equs}\sigma \otimes X^{\ell} \tau
\end{equs}
are set to be zero when one applies $ \left(   \varphi(\bigcdot)(y) \otimes \varphi^{-}(\bigcdot)(0) \right) $ to them. Indeed, $\varphi^{-}(X^{\ell})(0)$ is equal to zero.
\end{proof} 
Note the shift in notation due to the simplifications implied by the choice of parameters in the previous proposition. 
Then, for $ x=0 $, the character $ \varphi^{-} := \varphi^{-}(\bigcdot)(0) $ is given by
\begin{equs}
	\varphi^{-}(\widehat\tau)(0) 
	& = - \mathrm{T}_{|\widehat\tau|_{\s},0,0} \left(  \bar \varphi (\widehat\tau) \right) 
	= - \text{ev}_0 \left( \bar \varphi (\widehat\tau)  \right),
\end{equs}
where $ \text{ev}_x $ is the evaluation at the point $ x $.
These simplifications have been observed in \cite[Remark 2.11]{BS} and make sense when one can just look at $ \Pi_0 $ for the convergence of the model. This is true for a random model whose law is   invariant by translation, in the sense that $ \Pi_x(\widehat\tau) $ and $ \Pi_0(\widehat\tau) $ have the same law. 


\begin{remark}
The construction relies on $ \PPi  $ being a character. When in Regularity Structures, we renormalise characters with a suitable map $ M : \bar\CT^{\CD} \rightarrow \bar\CT^{\CD} $, we construct $ \PPi^{M} $ as a character on $ \bar\CT^{\CD}_+ $ but on $ \bar\CT^{\CD} $ it will not possible: we need to renormalise ill-defined distributional product and so the multiplicativity will be lost.
Therefore, the construction is still valid for $ \PPi^{M} $ viewed as a character on $ \bar\CT^{\CD}_+ $. Indeed, $ \PPi^{M} $ can be extended in the following way:
\begin{equs}
 \PPi^{M} X^k \prod_{i} \mathcal{J}_{(\Labhom_i,k_i)}(\tau_i) := \left( \PPi^{M} X^k \right) \prod_{i} \left( \PPi^{M} \mathcal{I}_{(\Labhom_i,k_i)}(\tau_i) \right)
\end{equs}
 Therefore, $ \varphi_- $ will still be a character but not  $ \varphi_{+} $. We will expand this construction in the next section.
\end{remark}

\subsection{Negative Renormalisation}

 Let $ \hat\CT_-^{\CD} $ be the free commutative algebra generated by $ \bar \CT^{\CD} $. We denote by $ \bigcdot $ the forest product associated to this algebra. The empty forest is given by $  \one_1 $. Elements of $ \hat\CT_-^{\CD} $ are of the form  $ (F,\Labn,\Labe) $ where $ F $ is a forest.
 The forest product is defined by:
 \begin{equs}
 	(F,\Labn,\Labe) \bigcdot (G,\bar \Labn,\bar \Labe)  
 	= (F \bigcdot G,\bar \Labn + \Labn,\bar \Labe + \Labe),
 \end{equs} 
where the sums $ \bar \Labn + \Labn $ and  $ \bar \Labe + \Labe $ mean that decorations defined on one of the forests are extended to the disjoint union by setting them to vanish on the other forest.
 Then we set $ \bar \CT_{-}^{\CD} = \hat \CT_{-}^{\CD} / \CB^{\CD}_{+} $ where $  \CB^{\CD}_{+}  $ is the ideal of $ \hat \CT^{\CD}_- $ generated by 
$ \lbrace \widehat\tau \in B_{\circ} \; : \; | \widehat\tau |_{\s}\geq 0 \rbrace $. Then, one defines in \cite{BHZ},  a coaction $ \hat{ \Delta}^{\!-} :  \bar\CT^{\CD} \rightarrow \bar\CT^{\CD}_- \otimes \bar\CT^{\CD} $ which is a deformation of an extraction-contraction coproduct in the same spirit as $ \hat{ \Delta}^{\!+} $. By multiplicativity, one extends it to a coaction $ \hat{ \Delta}^{\!-}: \hat \CT^{\CD}_- \rightarrow \bar\CT^{\CD}_- \otimes \hat \CT^{\CD}_- $.

Then one can turn this map into a coproduct $ \bar{ \Delta}^{\!-} : \bar\CT^{\CD}_-  \rightarrow \bar\CT^{\CD}_{-} \otimes \bar\CT^{\CD}_{-}$ and obtains a connected Hopf algebra for $ \bar\CT^{\CD}_{-}  $ endowed with this coproduct and the forest product see \cite[Prop. 5.35]{BHZ}.  The main difference here is that we do not consider extended decorations but the results for the Hopf algebra are the same as in \cite{BHZ}. 
The twisted antipode is given in \cite[Prop. 6.6]{BHZ}

\begin{proposition}
There exists a unique algebra morphism $\tilde\CA_-\colon \bar\CT^{\CD}_- \to \hat\CT^{\CD}_-$, that we call the ``negative twisted antipode'', such that for $\widehat\tau\in\bar\CT^{\CD}_- \cap\ker\one^\star_1$
\begin{equ}[e:pseudoant-]
	\tilde\CA_- \widehat\tau 
	= -\hat\CM_-(\tilde\CA_-\otimes\id)(\hat{ \Delta}^{\!-} 
	\mathfrak{i}_-\widehat\tau
	-\widehat\tau\otimes \one_1),
\end{equ}
where   $  \mathfrak{i}_-$ is the injection from $\bar\CT^{\CD}_-$ to $ \hat\CT^{\CD}_-$ and $ \hat \CM_- $ is the forest product between elements of $ \hat\CT^{\CD}_-$.
\end{proposition}

\begin{remark} 
The formalism described here corresponds to the one of  Remark~\ref{twisted_antipodeb} where $ H = \bar\CT^{\CD}_-$ and $ \hat H = \hat\CT^{\CD}_-$. The fact that we get a connected Hopf algebra and that we do not see any polynomials in the reduced coaction comes from the definition of $ \CB^{\CD}_+ $ which contains all the polynomials $ X^n $. Therefore, after quotienting by $ \CB_+^{\CD} $, they do not belong to 
$ \bar\CT^{\CD}_- $.
\end{remark}

It remains to describe the target space where the Birkhoff factorisation takes place.
We denote by $ \mathfrak{X} $ the space of stationary processes $ X : \Omega \rightarrow \CC^{\infty} $ over an underlying probability space $ (\Omega, \CF, P) $. This means that the laws of $ X(z) $ and $ X(0) $ are equal when $ X \in \mathfrak{X} $ and $ z \in \R^{d+1} $.  We assume that derivatives of elements in $ \mathfrak{X} $, computed at $0$
have moments of all orders. We define $ \mathfrak{N} = S(\mathfrak{X}) $ as the  symmetrised tensors over  the linear span of $ \mathfrak{X} $. Any element of $ S(\mathfrak{X}) $ is of the form 
\begin{equs}
f_1 \odot \cdots \odot f_n : =\frac{1}{n!} \sum_{\sigma \in \mathfrak{G}_n} f_{\sigma(1)} \otimes \cdots \otimes f_{\sigma(n)}
\end{equs}
where the $ f_i $ belong to $ \mathfrak{X} $.
We consider the following splitting:
\begin{equs}
	\mathfrak{N} = \mathfrak{N}_{-} \oplus  \mathfrak{N}_{+},
\end{equs}
where $ \mathfrak{N}_{-}  $ is the space of constants and $ \mathfrak{N}_{+} $ is the subspace of $ \mathfrak{N}  $ such that each element $ F $ satisfies:
    \begin{equs}
    \tilde{\mathbb{E}}(F)(0) = 0,
    \end{equs}
where $ \tilde{\mathbb{E}} : S(\mathfrak{X}) \rightarrow \R $ is defined on $ f_1 \odot \cdots \odot f_n$ by
 \begin{equs}
	 \tilde{\mathbb{E}}( f_1 \odot \cdots \odot f_n )(0) 
	 = \prod_{i=1}^{n} \mathbb{E}(f_i)(0)
 \end{equs}
 and $ \mathbb{E} $ denotes the expectation over the underlying probability space. Then we consider the characters $ \psi  : \hat \CT_-^{\CD} \rightarrow \mathfrak{N} $, $  \psi_{+} : \hat \CT_-^{\CD} \rightarrow \mathfrak{N} $ and  $  \psi_{-} : \bar\CT^\CD_{-} \rightarrow \mathfrak{N}_{-} $ as 
\begin{equs} \label{def_psi_-}
	\psi = \PPi , 
	\quad 
	\psi_{-} = \tilde{\mathbb{E}} ( \PPi \tilde \CA_{-} \bigcdot  )(0)  , 
	\quad 
	\psi_{+} =  \psi_{-} \star \psi = ( \psi_- \otimes \psi)\hat{ \Delta}^{\!-}  ,
\end{equs}
where $ \PPi $ is extended multiplicatively to
$ \hat \CT_-^{\CD} $. The choice of the point $ 0 $ in the definition of $ \psi_- $ is not arbitrary since we suppose that $ \PPi $ takes values in $ \mathfrak{X} $. One can show that $ \psi_+ $ is taking values in $ \mathfrak{N}_+ $ see \cite[Theorem 6.18]{BHZ}. 
\begin{theorem}
 For $ \widehat\tau \in \hat \CT_-^{\CD} $ and $ \widehat\tau_1  \in \bar\CT^{\CD}_- $, one can defined the same character $ \psi_- $ in \eqref{def_psi_-} as:
\begin{equs} \label{Birkhoff_negative} 
\begin{aligned}
	\psi_-(\widehat\tau_1 )  
	& =  - \tilde{\mathbb{E}}\left( \bar \psi( \mathfrak{i}_-(\widehat\tau_1 )) \right)(0), 
	\quad 
	\psi_+(\widehat\tau)  
	   = \left( \id  - \tilde{\mathbb{E}} \right) \left( \bar \psi(\widehat\tau) \right)(0) \\
	\bar \psi(\widehat\tau) 
	& =  \psi(\widehat\tau) +\sideset{}{^-} \sum_{(\widehat\tau)} \psi_{-}(\widehat\tau') \psi(\widehat\tau''). 
\end{aligned}
\end{equs}
where the Sweedler's notation is used for the modified reduced coaction:
\begin{equs}
	\Delta^{\! -}_{\scriptscriptstyle{\mathrm{red}}} \widehat\tau 
	= \sideset{}{^-}\sum_{(\widehat\tau)} \widehat{\tau}' \otimes \widehat{\tau}''.
\end{equs} 
\end{theorem}
\begin{proof}
 One can easily check that $  \tilde{\mathbb{E}}  $ is a Rota--Baxter map in the sense of the Remark~\ref{multiplicative} and then apply Proposition~\ref{Birkhoff_comodule}.
\end{proof}
\begin{remark}
The character $ \psi_{-} $ plays a central role in the definition of the renormalised model $ \hat \Pi_x $ which is given in \cite[Section 6.3]{BHZ} by:
\begin{equs}
	\hat \Pi_x = \left( \PPi M \otimes f_x M \right) \hat{ \Delta}^{\!+}, 
	\quad 
	M = \left( \psi_- \otimes \id \right) \hat{ \Delta}^{\!-}.
\end{equs}
The fact that this definition gives again a model relies on the cointeraction (see \cite[Theorem 5.37]{BHZ}) between the two Hopf algebras $ \hat\CT^{\CD}_+ $ and $ \bar\CT^{\CD}_- $ obtained when one adds extended decorations. This cointeraction has been observed on similar structures without any decorations in \cite{CA}.

\end{remark}

\end{document}